\documentclass[11pt]{amsart}

\usepackage{amsmath,amssymb,amsthm,amscd}
\usepackage[matrix,arrow, curve]{xy}
\numberwithin{equation}{section}
\usepackage{xcolor}
\usepackage{graphicx}
\usepackage[hypertexnames=false,linkcolor=black, colorlinks=true, citecolor=black, filecolor=black,urlcolor=black]{hyperref}

\newcounter{CountAlpha}

\theoremstyle{plain}
\newtheorem{MainThm}[CountAlpha]{Theorem}
\newtheorem{Thm}{Theorem}[section]
\newtheorem{Lem}[Thm]{Lemma}
\newtheorem{Prop}[Thm]{Proposition}
\newtheorem{Cor}[Thm]{Corollary}
\newtheorem{Hypothesis}[Thm]{Hypothesis}
\newtheorem{Claim}[Thm]{Claim}

\theoremstyle{definition}
\newtheorem*{MainDef}{Definition}
\newtheorem{Def}[Thm]{Definition}
\newtheorem{Rk}[Thm]{Remark}
\newtheorem{Ex}[Thm]{Example}
\newtheorem{Obs}[Thm]{Observation}

\newtheorem*{RkIntro}{Remark}

\newcommand{\hD}{\widehat{D}}
\newcommand{\hR}{\widehat{R}}
\newcommand{\hS}{\widehat{S}}
\newcommand{\hg}{\widehat{g}}
\newcommand{\hP}{\widehat{P}}
\newcommand{\hy}{\widehat{y}}
\newcommand{\hx}{\widehat{x}}
\newcommand{\hX}{\widehat{X}}
\newcommand{\hY}{\widehat{Y}}
\newcommand{\hf}{\widehat{f}}
\newcommand{\hJ}{\widehat{J}}

\newcommand{\cI}{\mathcal{I}}
\newcommand{\cL}{\mathcal{L}}
\newcommand{\cM}{\mathcal{M}}
\newcommand{\cP}{\mathcal{P}}
\newcommand{\cR}{\mathcal{R}}
\newcommand{\cS}{\mathcal{S}}
\newcommand{\cX}{\mathcal{X}}

\newcommand{\IA}{\mathbb{A}}
\newcommand{\IL}{\mathbb{L}}
\newcommand{\IK}{\mathbb{K}}
\newcommand{\IQ}{\mathbb{Q}}	
\newcommand{\IR}{\mathbb{R}}	
\newcommand{\IZ}{\mathbb{Z}}

\newcommand{\Eck}{\mathrm{Vert}}
\newcommand{\HS}{\mathrm{HS}}
\newcommand{\car}{\mathrm{char}}
\newcommand{\ord}{\mathrm{ord}}
\newcommand{\Dir}{\mathrm{Dir}}
\newcommand{\Rid}{\mathrm{Rid}}
\newcommand{\Spec}{\mathrm{Spec}}
\newcommand{\ini}{\mathrm{in}}
\newcommand{\cl}{\mathrm{cl}}
\newcommand{\gr}{\mathrm{gr}}
\newcommand{\Pol}{\mathrm{Pol}}

\newcommand{\wrt}{with respect to }

\newcommand{\poly}[3]{\Delta(  #1  ; #2 ; #3  ) }
\newcommand{\cpoly}[2]{\Delta(  #1  ; #2  ) }

\newcommand{\vp}{{\varphi}}

\newcommand{\fP}{{P}}
\newcommand{\fQ}{{Q}}
\newcommand{\hfP}{{\widehat{P}}}

\begin{document}
\title[\resizebox{4.5in}{!}{Characteristic polyhedra of singularities without completion II}]
{
Characteristic polyhedra of singularities without completion - Part II
}

\author[Vincent Cossart]{Vincent Cossart}
\address{Vincent Cossart\\
	Professeur \'em\'erite at 
	Laboratoire de Math\'emathiques LMV UMR 8100\\
Universit\'e de Versailles\\
45 avenue des \'Etats-Unis\\
78035 VERSAILLES Cedex \\
France}
\email{cossart@math.uvsq.fr}

\author[Bernd Schober]{Bernd Schober}

\address{Bernd Schober, 
	Laboratoire de Math\'emathiques LMV UMR 8100,
	Universit\'e de Versailles,
	45 avenue des \'Etats-Unis,
	78035 VERSAILLES Cedex,
	France\\
	and
	Institut f\"ur Algebraische Geometrie,
	Leibniz Universit\"at Hannover,
	Welfengarten 1,
	30167 Hannover,
	Germany}
\curraddr{Institut f\"ur Mathematik, Carl von Ossietzky Universit\"at Oldenburg, 26111 Oldenburg, Germany}
\email{bernd.schober@uni-oldenburg.de}

\thanks{The second named author was supported by Research Fellowships of the Deutsche Forschungsgemeinschaft (SCHO 1595/1-1 and SCHO 1595/2-1)
}


\begin{abstract}
	Hironaka's characteristic polyhedron is an important combinatorial object reflecting the local nature of a singularity.
	We prove that it 
	can be determined 
	without passing to the completion if the local ring
	is a G-ring and if
	additionally {either}
	it is Henselian, 
	or a certain polynomiality condition $ (\mathrm{Pol}) $ holds, 
	or a mild condition $(*) $ on the singularity holds. 
	For example, the latter is fulfilled if the residue field
	is perfect. 
\end{abstract}

\maketitle

\tableofcontents

\section*{Introduction}
\label{Intro}

Let $ ( R, M, k = R / M ) $ be a regular local ring, $ J \subset R $ {be} a non-zero ideal and $ ( u ) = ( u_1, \ldots, u_e ) $ {be} a 
regular $ R $-sequence which is a regular $ (R/ J)$-sequence.  
In \cite{HiroCharPoly}, Hironaka associates a polyhedron $ \cpoly Ju $ to this situation, the so called characteristic polyhedron of $ ( J ; u) $, which is an important tool for the study of singularities.
(We refer to section 1 for a detailed definition).
It appears in his proof for resolution of singularities of excellent hypersurfaces of dimension two 
\cite{HiroBowdoin}
and also in the generalization to the case of arbitrary two dimensional excellent schemes by Jannsen, Saito and the first named author \cite{CJS},
see also \cite{HomeworkDim2}. 
Moreover, in \cite{BerndThesis} the second {named} author shows that the invariant introduced by Bierstone and Milman in order to give a proof for constructive resolution of singularities in characteristic zero can be purely determined by considering certain polyhedra, which are closely connected to Hironaka's polyhedron and its projections.
In \cite{MS_qo}, Mourtada and the first named author provide a characterization of quasi-ordinary hypersurfaces 
	(defined over algebraically closed fields of characteristic zero) 
	which involves re-embedding of a singularity constructed via a weighted version of Hironaka's characteristic polyhedron.  
Furthermore, the characteristic polyhedron plays an essential role in recent work by Piltant and the first named author \cite{CP1}, \cite{CP2}, \cite{CPmixed} on the resolution of singularities of (arithmetic) threefolds.

\smallskip 

Let $ (y) = (y_1, \ldots, y_r) $ be a system of elements in $ R $ extending $(u) $ to a regular system of parameters for $ R $. 
Every $ g \in M $ has a finite expansion
\[ 
g = \sum_{(A,B) \in \IZ^{e + r }_{\geq 0 } } C_{A,B} u^A  y^B
\] 
with coefficients $ C_{A,B} \in R^\times \cup \{ 0 \} $ (see \cite{CPmixed} Proposition~2.1).
If we have $ g \notin \langle u \rangle $, then $ \nu := n_{(u)}(g) := \inf\{ |B| \mid C_{0,B} \neq 0  \} $ is a positive integer
and the \emph{polyhedron associated to $ ( g ,u, y ) $}, denoted by $ \poly guy $, is defined {as} the smallest closed convex set containing all points of the set
(which {may be} empty)
\[
\left\{
\;	\frac{A}{\nu - |B|} \;\bigg| \; C_{A,B} \neq 0 \, \wedge \, |B| < \nu
\;
\right\} + \IR_{\geq 0}^e.
\]	
Let $ ( f ) = (f_1, \ldots , f_m ) $ be a $ (u) $-standard basis for $ J $ 
(which is a particular system of generators, for which we have $ f_i \notin \langle u \rangle $ in particular, see Definition~\ref{Def:u-standardbasis}).
The \emph{polyhedron associated to $ ( f, u, y ) $}, denoted by $ \poly fuy $, is defined {as} the smallest closed convex set containing $ \bigcup\limits_{ i = 1 }^m \poly{f_i}{u}{y} $. 
One has $ \poly fuy \supseteq \cpoly Ju $ and a natural questions arising is whether it is possible to obtain equality.

Suppose $ ( y ) $ determines the directrix of 
{$ J' := J \cdot R / \langle u \rangle $} 
(which is a technical condition that we explain below in more details).
In this situation, Hironaka proves that one can modify a given $ (u) $-standard basis $ ( f ) = (f_1, \ldots , f_m ) $  for $ J $ 
and the system $ ( y ) $ 
in a systematic way such that 
for the resulting elements $ ( \hf; \hy ) $,
we have 
\[  
	\poly \hf u\hy = \cpoly Ju = \cpoly{J \hR}u .
\]
More precisely, $ ( \hf ) $ and $ ( \hy ) $ are constructed by the process of ``vertex preparation".  
This procedure consists of two parts which are applied alternately: 
normalization of given generators and solving vertices of the associated polyhedron.
While the first one concerns certain good choices of the generators, the latter are translations of the system $ ( y ) $.

In general, solving vertices is not finite, see Example~\ref{Ex:HiroInfinite}, 
and hence it may become necessary to pass to the completion of $ R $.
The {\em goal of the present article} is to investigate under which assumptions on $ (J, R, u ) $
we may determine the characteristic polyhedron without having to pass to the completion.

In \cite{CPcompl}, Piltant and the first named author show that one can attain Hironaka's characteristic polyhedron if $ R $ is a G-ring, $ J $ is principal, and $ r = 1 $.
(Note that a regular local ring $ R $ is a G-ring if and 
only if $ R $ is excellent, by \cite{CPcompl} Lemma~3.1,
which we recall in Lemma~\ref{lem:G=>excellent}).
This turns out to be an essential tool in their proof for resolution of singularities in dimension three (see \cite{CPmixed} Proposition 2.4).
Therefore having in mind the goal of constructing an embedded resolution of singularities it is natural to ask whether this is true in a more general situation.

\smallskip 

Let us provide more details for stating the precise results.
Let us fix some  technical notation.
Set 
\[	 
	R' := R / \langle u \rangle, \ \ \ 
	M' := M \cdot R' \ \ \
	\mbox{ and } \ \ \  
	J' := J \cdot R' .
\]
Let  $ ( y ) = ( y_1, \ldots, y_r ) $ be a system of elements in $ R $ extending $ ( u ) $ to a regular system of parameters for $ R $.
The associated graded ring of $ R '$ is 
\[
	\gr_{M'} (R') := \bigoplus_{ t \geq 0 } (M')^t / (M')^{t+1} \cong k [\overline{Y_1}, \ldots \overline{Y_r}],
\]
where $ \overline{Y_j} := \overline{y_j} \mod {M'}^2 $ and 
$ \overline{y_j} := y_j \mod \langle u \rangle \in R' $, for $ 1 \leq j \leq r $.

Consider $ g \in J $.
Denote by $ n_{(u)} ( g )  $ the order of $ \overline{g} = g \mod \langle u \rangle $ at the ideal $ M' $. 
We define $ \ini_{M'} ( \overline{g} ) := \overline{g} \mod {(M')}^{ n_{(u)}( g ) + 1 } $, if $ \overline{g} \not\equiv 0 $, and $ \ini_{M'} ( \overline{g} ) := 0 $ if $ \overline{g} \equiv 0 $.
Let $ C_{u} (J) $ be the tangent cone of $ J' $ \wrt $ M ' $, i.e.,  
$ C_{u} (J) $ is the cone defined by the homogeneous ideal 
\[ 
	\ini_{M'} ( J') := \langle \, \ini_{M'} ( \overline{g} ) \mid g \in J \, \rangle 
	\subset \gr_{M'} (R').
\] 

The {\em directrix} $\Dir(C)$ of the cone $ C := C_{u} (J)  \subset \Spec (\gr_{M'} (R')) \cong \IA^r_k $  
(or, more generally, of any cone)
is the biggest subvector space of $ \IA^r_k  $ leaving $ C $ stable under translation.  
Suppose we can choose the elements $ ( y ) $ above such that 
the ideal of the directrix $ \Dir(C) $ is $ \langle \overline{Y}  \rangle = \langle \overline{Y_1}, \ldots \overline{Y_r} \rangle $.
Then $ ( \overline{Y} ) $ is a {\em minimal} set of variables needed to write generators of the homogeneous ideal $ \ini_{M'} (J') \subset gr_{M'}(R')$,
i.e., there is no proper $ k $-submodule $ T \subset \gr_{M'}^1(R') $, $ T \neq \gr_{M'}(R') $, such that
\[
( \, \ini_{M'} (J')  \cap k [T ] \,) \,  gr_{M'}(R') = \ini_{M'} (J') .
\]
(Here, $ \gr_{M'}^1(R') $ denotes the part homogeneous of degree $ 1 $). 
In the latter case, we will also say {\em $ ( y ) $ determines the directrix of $ J' $}.

\begin{MainThm}[joint with O.~Piltant%
	\footnote{
		In the proof for the existence of the elements $ (z_1, \ldots, z_r) $ (section~\ref{sec:empty_hensel}), 
		we follow ideas outlined to us by Olivier Piltant during a private conversation on Theorems~\ref{MainThm:*} and~\ref{MainThm:Pol}.}]
	\label{MainThm:Hensel}
	Let $ (R,M,k) $ be a regular local G-ring, 
	$ J \subset R $ {be} a non-zero ideal, and $ ( u ) = ( u_1, \ldots, u_e ) $ be a 
	regular $ R $-sequence which is a regular $ (R/J)$-sequence.  
	Set $ R' = R / \langle u \rangle $ and $ J' = J \cdot R' $.
	Let $ ( y ) = (y_1, \ldots, y_r ) $ be a system of elements in $R$ extending $ ( u ) $ to a regular system of parameters for $ R $ and suppose 
	that the directrix of $ J'$ is determined by $ ( y ) $.

	Suppose that $ R $ is Henselian.
	Then there exist $ ( z ) = ( z_1, \ldots, z_r ) $ and $ ( g ) = ( g_1, \ldots, g_m ) $ in $ R $ such that $ ( u, z ) $ is a regular system of parameters~for $ R $, the system $ ( z ) $ yields the directrix of $ J' $, $ ( g ) $ is a vertex-normalized $ (u) $-standard basis of $ J $, and 
	\[
	\poly guz = \cpoly Ju  .
	\] 
\end{MainThm}

Another important object associated to a cone and closely related to the directrix, is the so called ridge:
The {\em ridge} $\Rid(C)$ ({\em fa\^\i te} in French) of the cone $ C $ is the biggest group of translations of 
$ \IA^r_{k} $
leaving $ C $ stable.
The ideal of the ridge $ \Rid(C) $ is generated by {\em additive homogeneous polynomials} $ \Phi_1,\ldots,\Phi_s \in k [ \overline{Y}_1, \ldots, \overline{Y}_r ]$,
and we have 
({see} \cite{BHM} section~2.2, {for example})
\[
( \, \ini_{M'} (J')  \cap k [ \Phi_1, \ldots, \Phi_s ] \,) \, gr_{M'}(R') = \ini_{M'} (J') ,
\]
where we require $ s $ to be minimal with this property. 
(Recall that a polynomial $ \Phi $ is called additive if $ \Phi (x+y) = \Phi(x) + \Phi (y) $).
After possible relabeling and linear changes in $ (\overline{Y}) $, we have   
$  
	\Phi_i 
= \overline{Y_i}^{q_i} + \sum\limits_{j=i+1}^r  \lambda_{i,j} \overline{Y_j}^{q_i}, 
$
where $ \lambda_{i,j} \in k \setminus k^{q_i} $ and $ q_i = p^{e_i } $,
for some $ e_i \in \IZ_{\geq 0 } $ 
with $ e_i \leq e_{i+1} $,
$ 1 \leq i < j \leq s $.

\smallskip

\begin{MainDef}[$\boldsymbol{*}$]
	\label{Def:(*)}
	We say that {\em condition $(*)$ holds (for $ (J,R,u) $)} if one of the following conditions is true:
	\begin{enumerate}
		\item[(a)]  the dimension of the ridge of $ C_u(J) $ coincides with the dimension of its directrix,
		\item[(b)] or $ \car ( k ) \geq \dfrac{\dim (X)}{ 2} + 1 $, 
			where $ X := \Spec(R/J) $.
	\end{enumerate}
\end{MainDef} 

Although assumption $(*)$ seems rather restrictive at first sight, 
	it includes a large class of ideals and rings;
	e.g., $(a)$ holds if the residue field $ k $ is perfect.

\begin{MainThm}
\label{MainThm:*} 
	Let $ (R,M,k) $ be a regular local G-ring, 
	$ J \subset R $ {be} a non-zero ideal, and $ ( u ) = ( u_1, \ldots, u_e ) $ be a 
	regular $ R $-sequence which is a regular $ (R/J)$-sequence.  
	Set $ R' = R / \langle u \rangle $ and $ J' = J \cdot R' $.
	Let $ ( y ) = (y_1, \ldots, y_r ) $ be a system of elements in $R$ extending $ ( u ) $ to a regular system of parameters for $ R $ and suppose 
	that the directrix of $ J'$ is determined by $ ( y ) $.
	
	Suppose that condition $ (*) $ holds.
	Then there exist $ ( z ) = ( z_1, \ldots, z_r ) $ and $ ( g ) = ( g_1, \ldots, g_m ) $ in $ R $ such that $ ( u, z ) $ is a regular system of parameters~for $ R $, the system $ ( z ) $ yields the directrix of $ J' $, $ ( g ) $ is a vertex-normalized $ (u) $-standard basis of $ J $, and 
	\[
		\poly guz = \cpoly Ju  .
	\] 
\end{MainThm}

\smallskip 

Hironaka obtains the elements $ ( \hy) $ determining the characteristic polyhedron
as a translation by elements $ \widehat\phi_1, \ldots, \widehat\phi_r \in \langle u \rangle \hR$,
i.e.,
$ \hy_j = y_{ j} + \widehat\phi_j $.
In general, the elements $ (z) $ in Theorem~\ref{MainThm:Hensel} and~\ref{MainThm:*} cannot be obtained as such a translation of $ (y) $, see Example~\ref{Ex:HiroInfinite}.
Hence, it is natural to ask under which condition, this is possible, i.e., $ z_j = y_j + \phi_j $, for $ \phi_i \in \langle u \rangle \subset R $ and $ 1\leq j \leq r $.
In \cite{CPcompl} Corollary~3.4, this question is discussed for principal ideals such that $ r = 1 $. 

In our general situation, we introduce the following which generalizes the hypothesis of \cite{CPcompl} Corollary~3.4:

\begin{MainDef}[$ \boldsymbol \Pol $]
	Let $ (S,N,k) $ be a regular local $G$-ring contained in $ (R,M,k) $ and with the same residue field.
	Let $ J \subset R $ be a non-zero ideal.
	Suppose $ (u) = (u_1, \ldots, u_e) $ is a regular system of parameters for $ S $
	such that $ (u) $ is a regular $ (R/J) $-sequence.
	Let $ (f) = (f_1, \ldots, f_m ) $ be a $ (u) $-standard basis 
	of $ J $. 
	We say {\em hypothesis $ (\Pol) $ holds for $ (J,R,S,f,u,y) $} if
	$ R = S[y_1,\ldots,y_r]_{M} $,
	and
	$ f_i\in S[y_1,\ldots,y_r] $ 
	with 
	$ \deg_{y}(f_i) = n_{(u)} (f_i) $,
	for $ 1 \leq i \leq m $, 
	and the directrix of $ J'$ is determined by $ ( y ) $.
\end{MainDef}

\begin{RkIntro}
	As the reader maybe observes, $ (\Pol) $ looks like a multi-variable ``Weiertrass prepared" condition, but be aware that we allow in $ f_i $ terms of the form $ c_B y^B $ with $ |B| = \deg_y (f_i) $ and where the coefficient $ c_B \in  S $ is not necessarily a unit.
\end{RkIntro}

The condition $(\Pol)$ naturally arises when studying singularities (Remark~\ref{Rk:pol_important}).
In Proposition~\ref{Prop:Pol_often_and_else_smaller_dim}, we show that $ (\Pol) $ is always fulfilled when considering initial forms along compact faces of a polyhedron associated to $ (f;u;y) $.

\begin{MainThm}
	\label{MainThm:Pol}
	Let $ (R,M,k) $ and $ (S,N,k) $ be regular local G-rings such that
	$ R = S[y_1,\ldots,y_r]_{M} $. 
	Let
	$ J \subset R $ be a non-zero ideal and $ ( u ) = ( u_1, \ldots, u_e ) $ be a regular system of parameters for $S $ which is a 
	regular $ (R / J) $-sequence.
	Let $ ( f ) = (f_1, \ldots, f_m ) $ be a $ ( u ) $-standard basis for $ J $.

	Suppose that $ (\Pol) $ holds for $ (J,R,S,f,u,y) $.
	Then there exist $ \phi_1, \ldots, \phi_r \in \langle u \rangle \subset S $ 
	and $ ( g ) = ( g_1, \ldots, g_m ) $ in $ R $ such that,
	if we define $ z_j := y_j + \phi_j $, for $ 1 \leq j \leq r $, 
	then $ ( u, z ) $ is a regular system of parameters for $ R $, 
	the system $ ( z ) $ yields the directrix of $ J' $, 
	$ ( g ) $ is a vertex-normalized $ (u) $-standard basis of $ J $, and 
	\[
	\poly guz = \cpoly Ju  .
	\] 
\end{MainThm}

\smallskip

The article is organized as follows:
After providing background on Hironaka's characteristic polyhedron, 
we explain the reduction from a non-empty to an empty characteristic polyhedron
in section~\ref{sec:non-empty}, adapting the method of \cite{CPcompl}.
Then, we discuss the important case Theorem~\ref{MainThm:Pol}.
In particular, we introduce a finite normalization procedure if hypothesis~$ (\Pol) $ holds.
In section~\ref{sec:empty*} (resp., section~\ref{sec:empty_hensel}),
we show
how to obtain suitable parameters $ ( z ) $ if $ \cpoly Ju = \varnothing $
and if additionally hypothesis $(*)$ holds (resp., if $ R $ is Henselian). 
After that, in section~\ref{sec:empty_generators}, 
we explain how to find appropriate generators for $ J $ once $ (z) $ are given.
Finally, we provide remarks on the general case 
(e.g.,~sections~\ref{sec:non-empty} and~\ref{sec:empty_generators} apply for any regular local $ G$-ring)
 and discuss further examples.
	 
\smallskip

 \noindent
\emph{Acknowledgement:} 
The authors thank Olivier Piltant and Anne Fr\"uhbis-Kr\"uger 
for inspiring discussions on the topic.
Bernd Schober thanks 
{Dale Cutkosky and}
Orlando Villamayor for interesting discussions and for {their} hospitality during {visits to Missouri and Madrid, respectively}.

%
%
%
%
%
%
%
%
%
%
%
%
%
%
%
%
%

\medskip

\section{Hironaka's Characteristic Polyhedron}

To begin with, let us recall the definition of Hironaka's characteristic polyhedron.
More detailed references are section 7 of \cite{CJS}, section 2.2 of \cite{BerndThesis}, or Hironaka's original work \cite{HiroCharPoly}.
Along this, we prove three new technical results 
(Lemmas~\ref{Lem:nzld_(u)stdbasis_unique_exp_and_numb_of_gen} and~\ref{Lem:Norm_implies_Mini_VC} and Corollary~\ref{Cor:Norm_implies_Mini}).
We introduce some notations. 
Let
\begin{itemize}
	\item	$ ( R,  M, k = R/M  ) $ be a regular local ring,
	\item $ J \subset R $ be a non-zero ideal contained in the maximal ideal $ M $, and
	\item $ ( u , y ) = ( u_1, \ldots, u_e; y_1, \ldots, y_r ) $ be a regular system of parameters~of $ R $.
\end{itemize} 
Furthermore, we set 
\[ 
	R' := R / \langle u \rangle,
	\ \ 
	M' := M \cdot R',
	\ \
	J' := J \cdot R',
	\ \
	\overline{y_j} := y_j \mod \langle u \rangle
	\
	(1 \leq j \leq  r).
\]

While the choice of $ ( u ) $ is fixed, we will consider different choices for the system $ ( y ) $.
For the definition of the characteristic polyhedron, $ R $ is not necessarily a G-ring 
and (in the variant that we present)
the partition of the regular system of parameters~is arbitrary. 
The interesting situation later will be when $ ( y ) $ is chosen such that 
it yields the directrix of $ J' $ and $(u)$ is a regular $ (R/J)$-sequence.

If we denote by $ U_i := u_i \mod M^2 $ (resp.~$ Y_j = y_j \mod M^2 $) 
the image of $ u_i $ (resp. $ y_j $) in the associated graded ring $ \gr_{M} ( R) $ of $ R $, 
then we have
\[
	\gr_{M} ( R ) := 
	\bigoplus_{t\geq 0} M^t / M^{t+1} 
	\cong k [U_1, \ldots, U_e, Y_1, \ldots, Y_r ] = k [U, Y].
\]
 For $ g \in J $, $ g \neq 0 $, the order 
 at $ M $ is 
 $ \ord_M(g) := \sup \{ a \in \IZ_{\geq 0 } \mid g \in M^a \} $.
 The {\em initial form of $ g $ \wrt $ M $} is defined as
 \[	
	 \ini_M ( g ) := g \mod M^{\ord_M (g)+1} \in \gr_M (R).
 \]
	The {\em tangent cone of $ J $ \wrt $ M $} is the cone 
 $ 
	 C (J)  \subset \IA^{e+r}_k = \Spec ( \gr_M ( R ) )
 $
 defined by the homogeneous ideal 
 \[ 
 \ini_M ( J ) := \langle \ini_M ( g) \mid g \in J \rangle \subset \gr_M ( R ) .
 \]

Similarly to above, using the notation $ \overline{Y_j} := \overline {y_j} \mod (M')^2 $, we obtain
\[
	\gr_{M'} ( R' ) 
	\cong k [\overline{Y_1}, \ldots, \overline{Y_r} ] = k [\overline{Y}].
\]
Let $ C_u(J) := C(J') \subset \IA^r_k $ be
the tangent cone of $ J' $ \wrt $ M' $.

\smallskip

\begin{Def}
\begin{enumerate}
	\item 	Let $ L : \IR^e \to \IR $ be a
			{\em  semi-positive linear form} on $ \IR^e $.
				This means there are $ a_i \in \IR_{ \geq 0 } $ 
			 such that $ L ( v_1, \ldots, v_e ) = \sum\limits_{ i = 1 }^e a_i \, v_i $ for $ v = ( v_1, \ldots, v_e ) \in \IR^e $ 
			and at least one $a_i>0$.
			Further, $ L$ is {called} {\em rational} if $ a_i \in \IQ_{ \geq 0 } $, {for} $1\leq i \leq e$.
			If all $ a_i $ are positive, then $ L $ is called a {\em positive linear form}.
			We set
			\[
				\Delta ( L ) := \{ v \in \IR^e \mid L ( v ) \geq 1 \}.
			\]  
			The set of semi-positive (resp.~positive) linear forms on $ \IR^e $ will be denoted by $ \IL_0 = \IL_0 ( \IR^e ) $ (resp.~$ \IL_+ = \IL_+ ( \IR^e ) $).
	\item	A subset $ \Delta \subset \IR^e_{\geq 0 } $ is called a \emph{rational polyhedron} if there are rational semi-positive linear forms $ L_1, \ldots, L_t \in \IL_0 ( \IR^e ) $, $ t < \infty $, such that 
			\[		
				\Delta = \bigcap_{i = 1 }^t \, \Delta ( L_i ) .
			\]
	\item	A point $ v \in \IR^e_{\geq 0 } $ is called a \emph{vertex} of a convex set $ \Delta \subset \IR^e_{\geq 0} $ if there exists a positive linear form $ L \in \IL_+ (\IR^e) $ such that 
			\[
				\{ w \in \IR^e \mid L ( w ) = 1 \} = \{ \, v \, \} .
			\]
			We denote the set of vertices of $ \Delta $ by $ \Eck ( \Delta ) $.
\end{enumerate}
\end{Def}

\begin{Def}
	\label{Def:asso_poly}
\begin{enumerate}
	\item
	Let $ g \in R $ be an element in $ R $, $ g \notin \langle u \rangle $.
	Since $ R $ is Noetherian and $ R \to \hR $ is faithfully flat, 
	we can expand $ g $ in a \emph{finite} sum
	\begin{equation}
	\label{eq:expansion}
		g = \sum_{(A,B) \in \IZ^{e + r }_{\geq 0 } } C_{A,B} \, u^A  y^B
	\end{equation}
	with coefficients $ C_{A,B} \in R^\times \cup \{ 0 \} $.
	Denote by $ \nu := n_{(u)} ( g ) $ the order of $ \overline{g} = g \mod \langle u \rangle $ in the ideal generated by $ \overline{ y_j } = y_j \mod \langle u \rangle $, $ j \in \{ 1, \ldots, r \} $.
	The \emph{polyhedron associated to $ ( g ,u, y ) $}, denoted by $ \poly guy $, is defined to be the smallest closed convex set containing all the points of the set
	\[
		\left\{
			\;	\frac{A}{\nu - |B|} \;\bigg| \; C_{A,B} \neq 0 \, \wedge \, |B| < \nu
			\;
		\right\} + \IR_{\geq 0}^e.
	\]	
	
	\item
	Let $ ( f ) = ( f_1, \ldots, f_m ) $ be elements in $ R $ with $ f_i \notin \langle u \rangle $ for all $ i $.
	The \emph{polyhedron associated to $ ( f, u, y ) $}, denoted by $ \poly fuy $, is defined to be the smallest closed convex set containing $ \bigcup\limits_{ i = 1 }^m \poly{f_i}{u}{y} $
\end{enumerate}	
\end{Def}

In general, there are many choices for (\ref{eq:expansion}).
But, as it is explained in \cite{HiroCharPoly} at the beginning of \S 2, if we fix the regular system of parameters~$ ( u, y ) $ and if we require the number of appearing exponents to be minimal, then the set $ \{ (A,B) \mid C_{A,B} \neq 0 \} $ is uniquely determined 
(and the corresponding $ \ini_M(C_{A,B})$ are unique, even though the set $ \{ C_{A,B} \}  $ may vary).

\smallskip

If we consider an ideal $ J \subset R $ and generators $ ( f_1, \ldots, f_m ) $, then the polyhedron $ \poly fuy $ depends on the choice of the generators:

\begin{Ex}
\label{Ex:PolyGenDependent}
Let $ R $ be a regular local ring with regular system of parameters $ (u_1, u_2, y_1, y_2) $.
Consider the ideal $ J = \langle f \rangle \subset R $, where 
\[ 
	( f ) = ( f_1, f_2 ) = (\, y_1^2 + u_1^3, \, y_2^3 + u_2^7 \,) .
\]
Clearly, the systems
\[ 
\begin{array}{l} 
	(g ) = ( g_1, g_2 ) = ( f_1, f_2 + f_1 ) = (\, y_1^2 + u_1^3, \, y_2^3 + u_2^7 + y_1^2 + u_1^3 \, ), 
\\[5pt] 
 ( h ) = ( h_1, h_2 ) = ( f_1, f_2 + u_2^2 f_1 ) = (\, y_1^2 + u_1^3, \, y_2^3 + u_2^7 + u_2^2 ( y_1^2 + u_1^3) \, ) 
\end{array}
\]  
both generate $ J $.
The vertices of the corresponding polyhedra are given by
$ \Eck ( \poly fuy ) = \{ \, ( \frac{3}{2}, 0 ) ; \, (0, \frac{7}{3}) \, \} $ 
and $ \Eck ( \poly guy ) = \{ \, ( \frac{3}{2}, 0 ) ; \, (0, \frac{7}{2}) \, \} $
and $ \Eck ( \poly huy ) = \{ \, ( \frac{3}{2}, 0 ) ; \, (0, 2)  \, \} $.
Therefore, we have 
\[ 
	\poly guy \subsetneq \poly fuy \subsetneq \poly huy . 
\] 
\end{Ex}

\smallskip

In order to get hands on this dependence, we have to recall Hironaka's notion of a $ ( u ) $-standard basis of an ideal $ J $.

\begin{Def}
\label{Def:with_L}	Let $ (R, M, k ) $ be a regular local ring with regular system of parameters $ (u, y ) = ( u_1, \ldots, u_e, y_1, \ldots, y_r ) $.
	 Consider $ g = \sum C_{A,B} u^A y^B \in R $ with an expansion as in (\ref{eq:expansion}).
	\begin{enumerate}
		\item[(1)] 	(\cite{CJS} Setup A)
					The \emph{$ 0 $-initial form} of $ g $ is defined as 
					\[ 
					\ini_0 ( g ): = \ini_0 ( g )_{(u,y)} 
					=  \sum_{\substack{B \in \IZ^{r}_{\geq 0}\\ |B|=n_{(u)} ( g )}} \overline{C_{0,B}} \, Y^B \in k [Y_1, \ldots, Y_r]	,
					\]
					where $ \overline{C_{0,B}} = C_{0,B} \mod M $.
		\item[(2)]	(\cite{CJS} Definition 6.2(2))
					Let $ L \in \IL_+ ( \IR^e ) $ be a positive linear on $ \IR^e $ and set 
					\[  
						v_L ( g ) := v_L ( g )_{(u,y)} := \min \{  L( A ) + |B| \mid C_{A,B} \neq 0 \}.
					\]
					We define
					\[
						\ini_L (g ) := \ini_L (g )_{(u,y)} := \sum_{\substack{ (A,B) \in \IZ_{\geq 0}^{e + r}\\  L( A ) + |B| = v_L ( g )}} \overline{C_{A,B} } \, U^A \, Y^B \in k [ U, Y ]
					\]
					with $ \overline{C_{A,B}} = C_{A,B} \mod M $.
		\item[(3)]	(\cite{CJS} Definition 6.5)
					Let $ ( f ) = ( f_1, \ldots, f_m ) $ be a system of non-zero elements in $ R $.
					A positive linear form $ L \in \IL_+ ( \IR^e ) $ is called \emph{effective for $ ( f, u, y ) $} if $ \ini_L ( f_i ) \in k [ Y ] $, for all $ i \in \{ 1, \ldots, m \} $.
	\end{enumerate} 
\end{Def}

\begin{Def}[\cite{HiroCharPoly} Definition (2.20)]
	\label{Def:u-standardbasis}
	Let $ J \subset R $ be a non-zero ideal and 
	$( u ) = ( u_1, \ldots, u_e ) $ be a 
	regular $ R $-sequence.
	Let $ ( f ) = ( f_1, \ldots, f_m ) $ be a system of non-zero elements in $ J $.
	\begin{enumerate}
		\item 
		We say that $ (f) $ is a {\em standard basis of $ J $}, if the system 
		$ ( \ini_M(f)) = ( \ini_M(f_1), \ldots, \ini_M(f_m) ) $ is a standard basis of $ \ini_M(J) $, i.e., 
		\begin{enumerate}
			\item $ \ini_M(J) = \langle \ini_M(f_1), \ldots, \ini_M(f_m) \rangle \subset \gr_M(R) $,
			\item $ n_1 \leq n_2 \leq \ldots \leq n_m $, if $ n_i := \ord_M ( f_i ) $, and
			\item for all $ i \geq 1 $, we have $  \ini_M ( f_i ) \notin \langle \ini_M ( f_1 ), \ldots,  \ini_M ( f_{ i - 1 } ) \rangle $.
		\end{enumerate}
	
	\smallskip 
	
		\item 
		$ ( f ) $ is called a \emph{$ ( u ) $-effective basis of $ J $}, if there exists a system of elements $ (y) = (y_1, \ldots, y_r ) $ extending $ (u) $ to a regular system of parameters and a positive linear form $ L \in \IL_+ ( \IR^e ) $ such that $ \ini_L ( f_i ) =  \ini_0 ( f_i ) \in k [Y] $, for all $ i \in \{ 1, \ldots,  m \} $, and 
		\[
			\ini_L ( J ) := \langle \ini_L ( g ) \mid g \in J \rangle  = \langle \ini_0 ( f_1 ), \ldots, \ini_0 ( f_m ) \rangle .
		\]
		
		\item 
		A $ ( u) $-effective basis $ (f) $ is called a \emph{$ ( u ) $-standard basis of $ J $}, if additionally, 
		$ (\ini_0(f)) $ is a standard basis of $ \ini_L(J) $.
		In particular,
	\begin{enumerate}
		\item $ \nu_1 \leq \nu_2 \leq \ldots \leq \nu_m $, if $ \nu_i := n_{(u)} ( f_i ) = \ord_M (  \ini_0 ( f_i ) ) $, and
		\item for all $ i \geq 1 $, we have $  \ini_0 ( f_i ) \notin \langle \ini_0 ( f_1 ), \ldots,  \ini_0 ( f_{ i - 1 } ) \rangle $.
	\end{enumerate}
\end{enumerate}
The pair $ (y, L ) $ is called a \emph{reference datum} of the $ (u) $-effective basis.
\end{Def}

Since $ \ini_L ( f_i ) =  \ini_0 ( f_i ) \in k [Y] $, for all $  i \in \{1, \ldots, m \} $, $ L $ is effective for $ ( f ,u, y ) $.
Furthermore, if $ ( f ) $ is a $ ( u ) $-standard basis, 
then we have $ f_i \notin \langle u \rangle $, 
since $ \ini_L(f_i ) \neq 0 $ is a non-zero element in $ k[Y] $, 
for $ 1 \leq i \leq m $.
Note that in the previous example the system $ ( g ) $ is not a $ ( u ) $-standard basis for $ J $.

Whenever we speak of a $ ( u ) $-standard basis $ ( f ) $ and there are elements $ ( y ) $ fixed, we implicitly assume that there exists a positive linear form $ L \in \IL_+ ( \IR^e ) $ such that $ ( y, L ) $ is a reference datum for $ ( f ) $.

\begin{Rk}\label{Rk:existencestandard}
Let  $ ( f ) = ( f_1, \ldots, f_m ) $ be a standard basis for a non{-}zero ideal $ J \subset R $. 
We can choose a regular system of parameters
$(u,y) = ( u_1, \ldots, u_e,y_1, \ldots, y_r ) $ of $ R $
such that $\ini_M(f_i) \in k[Y_1, \ldots, Y_r]$, for $1\leq i \leq m$. 
Then $(f)$ is a $(u)$-effective basis of $J$ with reference datum $ (y, L ) $ where $L(v_1,\ldots,v_e)=v_1+\ldots+v_e$.
 \end{Rk}

Concerning the existence of a $ ( u ) $-standard basis, Hironaka proved 

\begin{Thm}[\cite{HiroCharPoly} Lemma (2.23) and Theorem (2.24)]
	Let $ R $ be a regular local ring
	with maximal ideal $ M $
	and let $ (u) = (u_1, \ldots, u_e) $ be a regular $ R $-sequence with $ u_i \in M $.
	Let $ J \subset M$ be a non zero ideal in $ R$.
	The following conditions are equivalent:
	\begin{enumerate}
		\item 
		There exists a $ (u) $-standard basis of $ J $.
		
		\item 
		There exists a $ (u) $-effective basis of $ J $.
		
		\item 
		$ ( u ) $ is a regular $ (R/J) $-sequence.
		
		\item 
		$ \gr_{\langle u \rangle} (R/J) $ is a polynomial ring in $ e $ variables over $ R/(J + \langle u \rangle) $.

		\item 
		$ \langle u \rangle R \cap J = \langle u \rangle J $.
	\end{enumerate}
\end{Thm} 

\begin{Rk}
	\begin{enumerate}
		\item
		The hypothesis $ R $ be a Noetherian local ring
		(not necessarily regular) suffices to get (3) $ \Leftrightarrow $ (4) $ \Leftrightarrow $ (5), 
		see
		\cite{HiroCharPoly} Lemma~(2.23). 
		
		\item 
		Hironaka provided more equivalent conditions, but in order to avoid more technical definitions, we skip them here. 
	\end{enumerate}
\end{Rk}

Let us recall the following important result on $ ( u ) $-standard bases:

\begin{Thm}[\cite{CJS} Theorem 6.9]
\label{Thm:CJSreferencedatum}
	Let $ ( f ) = ( f_1, \ldots, f_m ) $ be a $ ( u ) $-standard basis for an ideal $ J \subset R $.
	If $ ( y) = ( y_1, \ldots, y_r ) $ is a system extending $ ( u ) $ to a regular system of parameters~for $ R $ and $ L \in \IL_+ ( \IR^e ) $ is a positive linear form  which is effective for $ ( f, u, y ) $, then $ ( y, L ) $ is a reference datum for $ ( f ) $.
\end{Thm}

\begin{Def}
	Let $ J \subset R $ be a non-zero ideal and 
	$ ( u ) = (u_1, \ldots, u_e ) $ be a system of elements as before.	
		Let $ ( y ) = ( y_1, \ldots, y_r ) $ be a system of elements extending $ ( u ) $ to a regular system of parameters~of $ R $.
		We define
		\[ 
			\poly Juy := \bigcap_{ (f) }  \poly fuy,
 		\]
		where the intersection runs over all possible $ ( u ) $-standard bases $ ( f ) $ 
		of $ J $ (in particular, there exists a positive linear form $ L \in \IL_+ ( \IR^e ) $ such that $ ( y, L ) $ is a reference datum for $ ( f ) $). 
		Further, we set
		\[
			\cpoly Ju := \bigcap_{(y)} \poly Juy,
		\]
		where the intersection ranges over all systems $ ( y ) $ extending $ ( u ) $ to a regular system of parameters~of $ R $.
		The polyhedron $ \cpoly Ju $ is called the \emph{characteristic polyhedron of $ J $ \wrt $ ( u ) $}.
\end{Def}

This is not Hironaka's original definition,
but the following result by Hironaka implies that the definitions coincide
in the relevant setting.

\begin{Thm}[Hironaka]
\label{Thm:Hironaka}
		Let $ R $ be a regular local ring,
		$ J \subset R $ be a non-zero ideal, and 
		$ ( u ) = ( u_1, \ldots, u_e ) $ be a 
		regular $ R $-sequence which is a regular $ (R/ J)$-sequence.  
		Set $ R' = R / \langle u \rangle $ and $ J' = J \cdot R' $.
		Let $ ( y ) = ( y_1, \ldots, y_r ) $ be a system of elements in $ R $ extending $ ( u ) $ to a regular system of parameters~of $ R $ and assume that $ ( y ) $ yields the ideal generating the directrix of $ J' $.
		 
		There exists a $ ( u )$-standard basis $ ( \hf ) = ( \hf_1, \ldots, \hf_m ) $ in $ \hR $ and a system of elements $ ( \hy ) = ( \hy_1, \ldots, \hy_r) $ such that $ ( u, \hy ) $ is a regular system of parameters~of $ \hR $, $ ( \hy ) $ determines the directrix of $ J' $ and 
		$
			\poly{\hf}{u}{\hy} = \cpoly Ju.
		$ 
\end{Thm}

\noindent 
This follows directly from \cite{HiroCharPoly} Theorems (3.17) and (4.8) that we recall 
as Theorems \ref{Thm:3.17} and \ref{Thm:4.8} below,
after introducing some notions.
We use here the fact $ \cpoly{\widehat{J}}u = \cpoly Ju $ (\cite{HiroCharPoly} Lemma (4.5)), where $ \widehat{J} = J \cdot \widehat{R} $.
Throughout the paper, we will make use of this without mentioning.
In the proof of the theorem one obtains $ ( \hf ) $ and $ ( \hy ) $ by applying the procedure of vertex preparation which consists of alternately normalizing the generators and solving the vertices of $ \poly fuy $. 
Let us recall these two processes.

We begin with normalization.
For this, we introduce the following total ordering $ \preceq $ on $ \IZ^r $: 
for $ A, B \in \IZ^r $, we define
\begin{equation}
\label{eq:def_prec_succ} 
	{A  \preceq B 
	 \ \ \stackrel{\mathrm{def.}}{\Longleftrightarrow} \ \ \ 
 	|A| < |B| \ \mbox{ or } \ 
 	(|A| = |B| \mbox{ and } A >_{lex} B),} 
\end{equation} 
where $ \leq_{lex} $ denotes the lexicographical ordering on $ \IZ^r $.
 
For a non-zero, homogeneous polynomial $ 0 \neq G = \sum \lambda_{B} \,Y^B \in k[Y_1, \ldots, Y_r ] $, 
we define
the \emph{(leading) exponent of $ G $} by
\begin{equation}
\label{eq:def_exp}
	\exp ( G ) := {\min_{\preceq}} \{ B \in \IZ^r \mid \lambda_{B} \neq 0  \}.
\end{equation} 
The \emph{(leading) exponent of a non-zero, homogeneous ideal $ I \subset k[Y_1, \ldots, Y_r ] $}
is defined as the collection
\[
	\exp ( I ) := \{ \exp ( G ) \mid 0 \neq G \in I \mbox{ homogeneous } \}.
\]

If $ ( f) = ( f_1, \ldots, f_m) $ is a $ ( u ) $-standard basis of the ideal
that they generate, then we abbreviate
$
	\exp(f_i) := \exp (\ini_0 (f_i) )
$
and
\begin{equation}
\label{eq:def_exp(f1...fm)}
	\exp ( \langle f_1, \ldots, f_{i} \rangle ) :=
	\exp ( \langle \ini_0 (f_1), \ldots, \ini_0(f_{i}) \rangle ),
\end{equation} 
for $ i \in \{ 1, \ldots, m \} $.
(Note that $ f_i \notin \langle u \rangle $ and hence $ \ini_0 ( f_i ) \neq 0 $).

\begin{Def}[\cite{CJS} Definition~7.11]
	\label{Def:poly_nlzd}
	Assume given $ G_1, \ldots, , G_m \in k[[U]][Y] =  k[[U_1, \ldots, U_e]][ Y_1, \ldots, Y_r] $,
		\[
			G_i = F_i(Y) + \sum_{|B| < n_i} P_{i,B}(U) Y^B,
			\ \ \
			(P_{i,B}(U) \in k[[U]]), 
		\]
		where $ F_i(Y) \in k[Y] $ is homogeneous of degree $ n_i $ and {$ P_{i,B}(U) \in \langle U \rangle $}.
		\begin{enumerate}
			\item 
			$(F_1, \ldots, F_m) $ is {\em normalized} if writing
			\[
				F_i (Y) = \sum_B C_{i,B} Y^B
				\ \ \
				\mbox{ with } C_{i,B} \in k,
			\]
			$ C_{i,B} = 0 $ if $ B \in \exp(F_1, \ldots, F_{i-1}) $ for $ i \in \{ 1, \ldots, m \} $.
			\item 
			$ (G_1, \ldots, G_m) $ is {\em normalized}
			if $ (F_1, \ldots, F_m ) $ is normalized and 
			$ P_{i,B} \equiv 0 $ if $ B \in \exp(F_1, \ldots, F_{i-1}) $ for $ i \in \{ 1, \ldots, m \} $.
		\end{enumerate}
\end{Def}

\begin{Def}
	\label{Def:nlzd_in_v}
	Let $ ( f ) = ( f_1, \ldots, f_m ) $ be a $ ( u ) $-standard basis 
	of an ideal $ J \subset R $.
	Set $  \nu_i := n_{ (u) } ( f_i ) $.
	Let $ f_i = \sum C_{A,B,i} \, u^A \, y^B $ be finite expansions as in (\ref{eq:expansion}) with $ C_{A,B,i} \in R^\times \cup \{ 0 \} $.
	\begin{enumerate}
	\item 
	$ ( f ) $ is called \emph{$ 0 $-normalized} if the corresponding system of $ 0 $-initial forms $ (\ini_0 ( f_1 ), \ldots, \ini_0 (f_m) ) $ is 
	normalized in the sense Definition~\ref{Def:poly_nlzd}(1) 
	(with respect to $(Y)=(Y_1,\ldots,Y_r)$).
	\item 
	Let $ v \in \Eck ( \poly fuy ) $ be a vertex of $ \poly fuy $.
				For $ 1 \leq i \leq  m $, the \emph{$ v $-initial form} of $ f_i $ is defined as
				\[
					\ini_v ( f_i ) := \ini_v ( f_i )_{(u,y)} := \ini_0 ( f_i ) + \sum \overline{C_{A,B,i}} \, U^A \, Y^B \in k [ U, Y ],
				\]
				where the sum ranges over those $ ( A,B ) \in \IZ^{ e + r } $ with $ \frac{A}{ \nu_i  - |B| } = v $, and $ \overline{ C_{A,B,i}} = C_{A,B,i} \mod M $.
				
				We say $ ( f ) $ is \emph{normalized at the vertex $ v $} if $ ( \ini_v ( f_1), \ldots, \ini_v ( f_m )) $ is normalized 
				in the sense of Definition~\ref{Def:poly_nlzd}(2)
				with respect to $(Y)$. 
				If $ ( f ) $ is normalized at every vertex of $ \poly fuy $ then we call $ ( f ) $ \emph{vertex-normalized} (\wrt $ ( u, y ) $). 
	\end{enumerate}
	Most of the time, if there is no possible confusion, we will skip the locutions 
	``with respect to $(u,y)$''
	and
	``with respect to $(Y)=(Y_1,\ldots,Y_r)$''.
\end{Def}

Let us point out:
although we mentioned it only in the last part of the definition, all these notions depend on the choice of the regular system of parameters~$ ( u, y ) $.

\begin{Lem}[\cite{HiroCharPoly} Lemma (3.14)]
	\label{Lem:0Nlz}
	Let $ ( f ) = ( f_1, \ldots, f_m ) $ be generators for $ J \subset R $.
	Let $ ( u, y ) 
		= ( u_1, \ldots, u_e,y_1, \ldots, y_r ) $ be a 
	regular system of parameters of $ R $
	and set $ \nu_i := n_{(u)}(f_i) $, for $ 1 \leq i \leq m $.
	Suppose $ (\ini_0 ( f_1 ), \ldots, \ini_0 (f_m) ) $ is a standard basis for the ideal that it generates in $ k [Y_1, \ldots, Y_r]$.
	
	There exist elements $ d_{i,j} \in \langle y \rangle^{\nu_i - \nu_j } \subset R $,
	for $ 1 \leq j < i \leq m $, 
	such that if we set $ g_1 = f_1 $ and $ g_i = f_i - \sum\limits_{j = 1}^{i-1} d_{i,j} f_j $, for $ 2 \leq i \leq m $, then
	\begin{enumerate}
		\item $ ( g_1, \ldots, g_m ) $ is $ 0 $-normalized (\wrt $ ( u, y ) $),
		\item $ n_{(u)} (g_i) = n_{(u)} (f_i) $, for $ i \in \{ 1, \ldots, m \} $,
		and $ \poly guy \subset \poly fuy $. 
	\end{enumerate} 
	Furthermore, if $ ( f) $ is a $ ( u ) $-standard basis, so is $ ( g ) $.
\end{Lem}

Let us point out that if $ (f) $ is a $ (u )$-standard basis, then the assumption on $ (\ini_0(f_1), \ldots, \ini_0(f_m) ) $ holds.
Hence, given any $ ( u ) $-standard basis, we can pass to a $ 0 $-normalized one without passing to the completion.

\begin{Lem}
	\label{Lem:nzld_(u)stdbasis_unique_exp_and_numb_of_gen}
	Let $ J \subset R $ be a non-zero ideal.
	Let $ ( f ) = ( f_1, \ldots, f_m ) $ and $ ( g ) = ( g_1, \ldots, g_\ell ) $ be two $ 0 $-normalized $ ( u ) $-standard bases for $ J $. 
	Then, we have 
	\[ 
	\ell = m 
	\ \ \ 
	\mbox{ and }
	\ \ \  
	n_{ ( u ) } ( f_i ) = n_{ ( u ) } ( g_i ) 
	\ 
	\mbox{ for every } i \in \{ 1, \ldots, m \} .
	\]
	Furthermore, if we assume 
	$ \exp(f_i) \preceq \exp(f_{i+1}) $ and $ \exp(g_j) \preceq \exp(g_{j+1}) $, for $ 1 \leq i < m $ and $  1 \leq j < \ell $.
	Then, we have 
	\[ 
	\exp ( f_i ) = \exp ( g_i ), 
	\ 
	\mbox{ for every } i \in \{ 1, \ldots, m \} .
	\]
\end{Lem}

\begin{Rk} 
Without loss of generality, we may assume that the condition 
	$ \exp(f_i) \preceq \exp(f_{i+1}) $, for $ 1 \leq i < m $, 
	holds true for a given $ (u)$-standard basis $ (f_1, \ldots, f_m) $ (after a possible reordering of the elements).
	Without explicitly mentioning it, we assume from now on that 
	the leading exponents of a given $ ( u)$-standard basis are ordered increasingly \wrt $ \preceq $.
\end{Rk} 

\begin{proof}[Proof of Lemma~\ref{Lem:nzld_(u)stdbasis_unique_exp_and_numb_of_gen}]
The first assertion follows from the fact that, for any $(i,j)$, $1\leq i \leq j\leq m$ such that $ n_{(u)} (f_{i-1})<n_{(u)} (f_{i})=n_{(u)} (f_{j})<n_{(u)} (f_{j+1})$, then $ (\ini_0 ( f_i), \ldots, \ini_0 (f_j))$ is a basis of the $k$-vector space $\ini_{M'}(J')_s $modulo $\ini_{M'}(J')_{s-1}$ with $s:=n_{(u)} (f_{i})$, $ R' = R / \langle u \rangle $, $M'=M / \langle u \rangle $ and $ J' = J \cdot R' $.

	Let $ L_f : \IR^e \to \IR $, $ L_f ( v ) = \sum\limits_{i=1}^e \lambda_i v_i $ with $ \lambda_i \in \IR_+ $, be a positive linear form such that $ (y, L_f ) $ is a reference datum for $ ( f ) $ 
	and let $ L_g : \IR^e \to \IR $, $ L_g ( v ) = \sum\limits_{i=1}^e \mu_i v_i $ with $ \mu_i \in \IR_+ $, one such that $ ( y, L_g ) $ is a reference datum for $ ( g ) $, where $ v  = ( v_1, \ldots, v_e ) \in \IR^e $.
	In particular, $ L_f $ is effective for $ ( f ,u, y ) $ and $ L_g $ is effective for $ ( g ,u, y ) $. 
	
	Set $ \rho_i := \max \{ \lambda_i, \mu_i \} $, for $ 1 \leq i \leq e $, and define $ L : \IR^e \to \IR $ by $ L ( v ) = \sum\limits_{i=1}^e \rho_i  v_i $.
	Then $ L $ is effective for both  $ ( f, u, y ) $ and $ ( g, u, y ) $, and
	Theorem~\ref{Thm:CJSreferencedatum} implies that $ ( y, L ) $ is a reference datum for $ ( f ) $ as well as for $ ( g ) $.
	Thus,
	$$
	\langle \ini_0 ( f_1), \ldots, \ini_0 (f_m) \rangle =  \ini_L ( J ) =\langle \ini_0 ( g_1), \ldots, \ini_0 (g_\ell) \rangle 
	$$
	
	Suppose $ \exp ( f_1 ) \neq \exp ( g_1 ) $; 
	without loss of generality, we then have 
	$ \exp ( f_1 ) \prec \exp ( g_1 ) $.
	This contradicts 
	$ \ini_0 ( f_1) \in \langle \ini_0 ( g_1), \ldots, \ini_0 (g_\ell) \rangle $.
	Hence, we have $ \exp ( f_1 ) = \exp ( g_1 ) $.
	
	Assume $ \exp ( f_i ) = \exp ( g_i ) $, for all $ i < j $, and $ \exp ( f_j ) \neq \exp ( g_j ) $, for some $ j \geq 2 $; 
	without loss of generality $ \exp ( f_j ) \prec  \exp ( g_j ) $.
	Since we do have $ \ini_0 ( f_j) \in \langle \ini_0 ( g_1), \ldots, \ini_0 (g_\ell) \rangle $
	and $ \ini_0(f_j) \in k[Y] $,
	there is $ \mu_1, \ldots, \mu_\ell \in k [Y] $ such that 
	\[
	\ini_0 ( f_j) = \mu_1 \cdot \ini_0 ( g_1) + \ldots +  \mu_\ell \cdot \ini_0 ( g_\ell).
	\]
	As $ \exp ( f_j ) \prec  \exp ( g_j ) \prec  \exp ( g_{j+1} ) \prec  \ldots \prec  \exp ( g_\ell ) $ there must exist $ i < j $ with $ \mu_i \neq 0 $. 
	This implies that there appear $ g_i $ with $ \exp ( g_i ) = \exp ( f_i) $, for some $ i < j $.
	But this contradicts the property that $ ( \ini_0 (f) ) $ is normalized. 
	
	Assume that the $ 0 $-normalized $ ( u ) $-standard bases are of different length; without loss of generality, $ m > \ell $.
	Then $ 0 \neq \ini_0 (f_{ \ell + 1 }) \in \langle \ini_0 ( g_1), \ldots, \ini_0 (g_\ell) \rangle $.
	Since $ \exp ( g_i ) = \exp ( f_i ) $, for all $ 1 \leq i \leq \ell $, we obtain a contradiction to the hypothesis that $ ( \ini_0 (f) ) $ is normalized, as before.  
\end{proof}

Recall the following notion of \cite{HiroCharPoly} (2.3) and (2.6) (p.~260).

\begin{Def}
	\label{Def:I(Delta,b)} 
	Let $ R $ be a regular local ring. 
	We fix a regular system of parameters $ (u,y) = ( u_1, \ldots, u_e; y_1, \ldots, y_r ) $. 
	Let $ \Delta \subset \IR^e_{\geq 0} $ be a closed convex subset such that $ \Delta + \IR^e_{\geq 0} = \Delta $ and $ b \in \IR_+ $. 
	Set $ b\Delta := \{ bv \in \IR^e_{\geq 0 } \mid v \in \Delta \} $ and $ W(b) := \{ B \in \IR^r_{\geq 0 } \mid |B| \geq b \} $.
	\begin{enumerate}
		\item 	
		The symbol $ \{ \Delta, b \} $ denotes the smallest convex subset of $ \IR^{e+r}_{\geq 0} $ 
		containing 
		$ 
		(b\Delta) \times \IR^r_{\geq 0} 
		$ 
		and 
		$
		\IR^e_{\geq 0 } \times W(b).
		$ 
		
		\item 
		The symbol $ I (\Delta,b) := I (\Delta,b)_{(u,y)}:= I (\{\Delta,b\} )_{(u,y)} $ 
		denotes the monomial ideal in $ R $ generated by 
		$	
		\{ u^A y^B \mid (A,B) \in \{\Delta,b\} {\, \cap \, \IZ^{e+r}_{\geq 0}} \} .
		$
	\end{enumerate}	
\end{Def}

\begin{Rk}\label{Rk:IDelta} With notations as above, it is easy to verify: 
	\[ g\in I (\Delta,b) \Leftrightarrow 
		v_L ( g )\geq b, \hbox{ for all } L \in \IL_+ ( \IR^e ), \hbox{with } L(v)\geq 1 
			\hbox{ for every } v\in \Delta.
	\]
	
\end{Rk}

\begin{Prop}[\cite{HiroCharPoly} Lemma (3.15)]
\label{Prop:Nlz}
	Let $ ( f ) = ( f_1, \ldots, f_m ) $ be a $ 0 $-normalized system of generators for $ J \subset R $
	and let $ ( u, y ) $ be a regular system of parameters with $(u)$ a regular $ (R / J)$-sequence.
	Set $ \nu_i := n_{(u)}(f_i) $, for $ 1 \leq i \leq m $.
	Let $ v \in \IR^e_{\geq 0 } $ be a vertex of $ \poly fuy $.
	
	There exist $ x_{i,j} \in \langle u \rangle \cap I (v + \IR_{\geq 0 }^e , \nu_i - \nu_j )_{(u,y)} \subset R $ such that if we set $ g_1 = f_1 $ and $ g_i = f_i - \sum\limits_{j = 1}^{i-1} x_{i,j} f_j $, for $2 \leq i \leq m $, then we have 
	\begin{enumerate}
		\item	$ \poly guy \subset \poly fuy $,
		\item	$ ( g_1, \ldots, g_m ) $ is normalized at $ v $ if $ v $ is a vertex of $ \poly guy $, and 
		\item	$ \Eck ( \poly fuy ) \setminus \{ v \} \subset \Eck ( \poly guy ) $.
	\end{enumerate} 
\end{Prop}

\begin{Rk}\label{Rk:ini_0} 
	As $ x_{i,j} \in \langle u \rangle $, we have $\ini_0(g_i)=\ini_0(f_i)$ for $1 \leq i \leq m $, Hironaka's process preserves the extra assumption $ \exp(f_i) \preceq \exp(f_{i+1}) $, for $ 1 \leq i < m $.
\end{Rk}

\noindent 
In Example \ref{Ex:Preparation} below we show how normalization can eliminate vertices.

\begin{Lem}
	\label{Lem:Norm_implies_Mini_VC}
	If $ ( f  ) = ( f_1, \ldots, f_m ) $ and $(g)=( g_1, \ldots, g_{m'} ) $ are two \emph{vertex-normalized} $ ( u ) $-standard basis of $ J $ with reference data respectively $ ( y, L_1 ) $ and $ ( y, L_2 ) $, for some $ L_i\in \IL_+ $, $i=1,2$, then $m=m'$ and
	\begin{equation}
	\label{eq:Norm_implies_Mini_VC1}
	\poly{ f }{ u }{ y } = \poly{ g }{ u }{ y }.
	\end{equation}
\end{Lem}

\begin{proof}
	By Lemma \ref{Lem:nzld_(u)stdbasis_unique_exp_and_numb_of_gen}, both have the same number of elements, $ ( f ) = ( f_1, \ldots, f_m ) $ and $ ( g ) =  ( g_1, \ldots, g_m ) $.
	Furthermore $ n_{(u)} ( f_i ) = n_{(u)} ( g_i ) =: \nu_i $, for every $ i \in \{ 1, \ldots, m \} $. 
	
	Suppose equality \eqref{eq:Norm_implies_Mini_VC1} is wrong. 
	Then the convex hull $\Delta \subset \IR^e_{\geq 0 }  $ of 
	$ \poly{ f }{ u }{ y } \cup \poly{ g}{ u }{ y }$ 
	would be strictly greater than one of the associated polyhedra: 
	say $\poly{ g}{ u }{ y } \subsetneq \Delta$. 
	Hence, there would exist a vertex $v$ of $\Delta$ with $ v\notin \poly{ g}{ u }{ y } $ and a positive linear form $ \Lambda $ such that 
	$
	v = \Delta\cap \{ w\in \IR^e \mid \Lambda(w)=1   \}.
	$
	Therefore, we have 
	\[ 
	\begin{array}{c} 
	v=\poly{ f }{ u }{ y } \cap \{ w\in \IR^e \mid \Lambda(w)=1   \}
	\ \ \mbox{ and}
	\\[5pt]
	\poly{ g}{ u }{ y } \subset \{ w\in \IR^e \mid \Lambda(w)>1   \}.
	\end{array}
	\] 
	By Theorem \ref{Thm:CJSreferencedatum}, $ ( y, \Lambda ) $ is a reference datum for $(g)$.
	
	Let $i \in \{ 1, \ldots, m \} $ be minimal such that 
	$ F_i:= \ini_{\Lambda}(f_i)\neq \ini_0(f_i)$.
	In particular, $ \ini_{\Lambda}(f_i) \notin K[Y] $. 
	Since $ (y; \Lambda) $ is a reference datum for the $ ( u) $-standard basis $ (g) $,
	Definition~\ref{Def:u-standardbasis}(1) provides that
	$ F_i = \sum\limits_{\nu_j\leq \nu_i} \lambda_j \ini_{\Lambda}(g_j)$,
	for some $\lambda_j \in  \gr_{\Lambda}(R) \cong k[Y,U]$, $\lambda_j = 0$ or quasi-homogeneous for $ {\Lambda}$ of degree $ \nu_i- \nu_j$. 
	As $ \ini_{\Lambda}(g_j)=  \ini_0(g_j)\in K[Y] $ and 
	$\ini_{\Lambda}(f_i)\notin K[Y] $, 
	we obtain that $F_i-\ini_0(f_i) = \sum\limits_{\nu_j< \nu_i} \lambda_j \ini_{\Lambda}(g_j)\neq 0 $.
	This contradicts the assumption that $F_i$ is $v$-normalized.
\end{proof}

Let $ ( f ) = ( f_1, \ldots, f_m ) $  be any $ ( u ) $-standard basis for $ J $.
By Hironaka's procedure of normalization,  
we obtain a vertex-normalized $ ( u ) $-standard basis $ ( f' ) $ (possibly in $ \hR $) and
$ \poly{ f' }{ u }{ y } \subseteq \poly{ f }{ u }{ y } $
(Lemma~\ref{Lem:0Nlz} and Proposition~\ref{Prop:Nlz}).
Combining this with Lemma~\ref{Lem:Norm_implies_Mini_VC}, we get

\begin{Cor}
	\label{Cor:Norm_implies_Mini}
	If $ ( f  ) = ( f_1, \ldots, f_m ) $ is a \emph{vertex-normalized} $ ( u ) $-standard basis of $ J $ with reference datum $ ( y, L ) $, for some $ L \in \IL_+ $, then
	$$
	\poly{ f }{ u }{ y } = \poly{ J }{ u }{ y }.
	$$ 
\end{Cor}

The previous result shows that for an arbitrary $ ( u ) $-standard basis $ ( f ) $ of $ J $ the difference between $ \poly fuy $ and $ \poly Juy $ reflects how far $ ( f ) $ is away from being vertex-normalized.

\smallskip

After normalizing the generators one has to check whether vertices of the associated polyhedron can be eliminated by changes in the elements $ ( y ) $.

\begin{Def}
	Let $ ( f ) = ( f_1, \ldots, f_m ) $ be generators of an ideal $ J \subset R $.
	Let $ ( u ,y ) $ be a regular system of parameters for $ R $ such that $ ( y ) $ determines the directrix of $ J' = J \cdot R' $, where $ R' = R / \langle u \rangle $.
	A vertex $ v \in \poly fuy $ is called \emph{solvable} if 
	$ v \in \IZ^e_{\geq 0} $ and 
	there exist $ \lambda_j \in R^\times \cup \{ 0 \} $, $ j \in \{ 1, \ldots, r \} $, such that we have, for the system $ ( z ) = (z_1, \ldots, z_r ) $, given by $  z_j := y_j - \lambda_j u^v $,
	\[
		v \notin \poly fuz .
	\] 
	The elements $ (\lambda_1 u^v, \ldots, \lambda_r u^v) $ are then called a {\em solution for $ v $} or 
	simply {\em $ v $-solution}.
\end{Def}

In \cite{HiroCharPoly}, Corollary (4.4.3) it is shown that,
as $ ( y ) $ determines the directrix of $ J' $,  
 if a vertex is solvable, then the residues $ \lambda_j \mod M \in k $ are unique.

Note that $ ( z ) $ has still the property that it yields the directrix of $ J' $.
Moreover, the other vertices of the polyhedron do not change under this translation.
More precisely,

\begin{Prop}[\cite{HiroCharPoly} Lemma (3.10)]
\label{Prop:SolVert}
	Let $ ( f ) $, $ J \subset R $ and $ ( u ,y ) $ be as in the previous definition.
	Let $ v \in \poly fuy $ be a solvable vertex 
	with $ v $-solution $ (\lambda_1 u^v, \ldots, \lambda_r u^v ) $.
	Define $ ( z ) = (z_1, \ldots, z_r) $ by $ z_j = y_j - \lambda_j u^v $, for $ 1 \leq j \leq r $.
	We have
	\begin{enumerate}
		\item	$ \poly fuz \subset \poly fuy $,
		\item	$ v \notin \poly fuz $, and 
		\item	$ \Eck ( \poly fuy ) \setminus \{ v \} \subset \Eck ( \poly fuz ) $.
	\end{enumerate}
\end{Prop}

In order to normalize and to solve the vertices in a systematic way one has to equip $ \IR^e $ with a total ordering.
Then one picks the vertex that is minimal \wrt this ordering, normalizes, and tests if it is solvable. 
After that one takes the new minimal vertex that has not been considered yet.

In the procedure it is important to apply alternately normalization and vertex solving.
In the latter we only take care of vertices and not points in the interior of the polyhedron.
But still these points might be interesting after normalization.

\begin{Ex}
	\label{Ex:Preparation}
	Let $ R $ be a regular local ring with regular system of parameters $ (u_1, \ldots, u_e, y_1, y_2) $.
	Assume that $ R $ contains a field of characteristic $ p > 0 $.
	Consider the polynomials $ f_1 = y_1^{ p } $ and $ f_2 = y_2^{ p^2 } + u^{ A' } y_1^{ p } + u^{ p^2 A } $ and
	suppose that $ A \in \IZ_{\geq 0}^e \cap (\frac{A'}{p^2 - p} + \IZ_{\geq 0}^e) $, $ A \neq \frac{A'}{p^2 - p} $.
	The only vertex of the associated polyhedron is given by $ v := \frac{A'}{p^2 - p} $ and one sees easily that $ v $ cannot be solved.
	The normalization provides $ g_2 := f_2 - u^{A'}\hspace{-4pt}\cdot\hspace{-2pt} f_1 = y_2^{ p^2 } + u^{ p^2 A } $.
	Therefore the vertex $ v $ vanishes and the new vertex $ A $ is solvable via $ z_2 := y_2 + u^A $.
\end{Ex}

\begin{Ex}
	\label{Ex:HiroInfinite}
	Let $ R $ be a regular local ring with regular system of parameters $ (u_1, u_2, y ) $.
	Suppose $ R $ contains a field of characteristic $ p > 0 $. 
	Consider the singularity given by 
	\[ 
	f = y^p + y^{p^2} + u_1^{a p} + u_2^b = 0,
	\]
	for $ a, b \in \IZ_+ $ prime to $ p $.
	The initial form at the vertex $ v_0 := (a,0) $ is $ \ini_{v_0} (f) = Y^p + U_1^{ap} = (Y + U_1^a)^p $. 
	Following the above procedure we have to make the translation $ y \mapsto w := y + u_1^a $ and get \[ 
		f = w^p + w^{p^2} - u_1^{a p^2}  + u_2^b.
	\]
	The new vertex $ v_1 := (ap , 0) $ is solvable 
	and clearly this is \emph{not} a finite process.
	On the other hand, if we consider $ z := y + y^p + u_1^a $, then $ f = z^p + u_2^b $ and the associated polyhedron coincides with the characteristic polyhedron.
\end{Ex}

For another example, which is valid in a more general setting, we refer the reader to Example 3.1 in \cite{CPcompl}, p.~165.

\begin{Hypothesis}\label{hyp}
From now on, we always assume the following:
\begin{enumerate}
		\item	 $ (u, y ) $ is a  regular system of parameters of $ R $,
		\item	$ u $ is a regular $ (R/ J) $-sequence, and 
		\item	there is no proper $ k $-submodule $ T \subset \gr_{M'}^1(R') $ 
		such that
\[
( \, \ini_{M'} (J')  \cap k [T ] \,) \,  gr_{M'}(R') = \ini_{M'} (J'),
\]
 with $ R' = R / \langle u \rangle $ and $ J' = J \cdot R' $, $J\subset R$ a non zero ideal.
	\end{enumerate}
Even if some statements below are true with less restrictive hypotheses.
\end{Hypothesis}

\begin{Def}[\cite{CJS} Definition 7.15(1),(4)] 
	Let $ J \subset R $ be a non-zero ideal and $ ( u,y ) = (u_1, \ldots, u_e ,y_1, \ldots, y_r ) $ be a system of elements  a regular system of parameters of $ R $ verifying Hypothesis~\ref{hyp}.
	\begin{enumerate}
		\item 
		Let $ v \in \poly fuy $ be a vertex.
		We say $ (f;y) $ is {\em $ v $-prepared \wrt $ (u) $} if
		$ (f) $ is normalized at $ v $ and $ v $ is not solvable. 
		
		\item 
		If $ (f;y) $ is $ v $-prepared \wrt $ ( u) $ at every vertex of $ \poly fuy $, then 
		$ (f) $ is called {\em well prepared \wrt $ (u) $}. 
	\end{enumerate}
\end{Def}

\begin{Thm}[\cite{HiroCharPoly} Theorem (3.17)]
	\label{Thm:3.17}
	Let $ R $ be a regular local ring and let  
	$ ( u,y ) = (u_1, \ldots, u_e; y_1, \ldots, y_r ) $ and $J \subset R $ verify Hypothesis~\ref{hyp}. 	
	Let $ (f) = (f_1, \ldots, f_m) $ be a system generating $J$ such that $ f_i \notin \langle u \rangle $, for $ 1 \leq i \leq m $.
	Set $ \nu_i := n_{(u)}(f_i) $.
	Let us assume that $ R $ is {\em complete} and that 
	$ \ini_0 (f)_{(u,y)} $ is a standard basis of the ideal which it generates in $ \gr_M(R) $. 
	 
	There exist $ x_{i,j} \in  I (\poly fuy; \nu_i - \nu_j )_{(u,y)} $, for $ 1 \leq j < i \leq m $,
	and $ d_\alpha \in I(\poly fuy; 1)_{(u,y)} \cap \langle u \rangle \subset R $, for $ 1 \leq \alpha \leq r $, such that
	$(g;z) $ is totally prepared \wrt $ (u) $,
	where $ (g) = (g_1, \ldots, g_m) $
	with $ g_i = f_i - \sum\limits_{j=1}^{i-1} x_{i,j} f_j $
	and $ (z) = (z_1, \ldots, z_r) $ with $ z_\alpha = y_\alpha - d_\alpha $.

	Moreover, if $ (f) $ is $ 0 $-normalized (\wrt $ (u,y) $), then we can choose those $ x_{i,j} $ in the ideal $ I (\poly fuy; \nu_i- \nu_j)_{(u,y)} \cap \langle u \rangle \subset R $, {cf. Remark~\ref{Rk:ini_0}}. 
\end{Thm}

\noindent
The notion of $ (f;y) $ to be totally prepared 
which is slightly more restrictive than being well prepared 
(see \cite{CJS} Definition 7.15 (5)). 
Nonetheless, the polyhedron does not change if we pass from a well prepared to a totally prepared $ ( u ) $-standard basis, so we can avoid introducing more technical notions.

In the original formulation, there is also another additional remark on the choice of $ x_{i,j} $ and $ d_\alpha $ if only a finite set of vertices of $ \poly fuy $ is considered.
Since this is not needed in our context, we skipped it.  


\begin{Thm}[\cite{HiroCharPoly} Theorem (4.8)]
	\label{Thm:4.8} 
	Let $ R $ be a regular local ring and let  
	$ ( u,y ) = (u_1, \ldots, u_e; y_1, \ldots, y_r ) $ and $J \subset R $ verify Hypothesis~\ref{hyp}. 	
 	Further, let $ (f) = (f_1, \ldots, f_m) $ be a $(u)$-standard basis of $J$.

	If $ v \in \poly fuy $ is any vertex such that $ (f;y) $ is $ v $-prepared \wrt $ ( u ) $,
	then $ v $ is also a vertex of $ \cpoly Ju $. 
	In particular, if $ (f;y) $ is well prepared \wrt $ (u ) $, then $ \poly fuy = \cpoly Ju $. 
\end{Thm}

%
%
%
%
%
%
%
%
%
%

\medskip 

\section{Non-empty Characteristic Polyhedron}
\label{sec:non-empty}

We first show how to reduce the problem to the case of an empty characteristic polyhedron. 
Here, the assumption on $ R $ being a G-ring appears for the first time. 
Thus, let us recall the following result of \cite{CPcompl}:

\begin{Lem}[\cite{CPcompl} Lemma 3.1]
	\label{lem:G=>excellent}
	Let $ ( R,  M, k ) $ be a regular local ring.  
	Then 
	\begin{center}
		$ R $ is a G-ring \ \ $ \iff $ \ \ $ R $ is quasi-excellent
		\  \ $ \iff $ \ \ $ R $ is excellent.
	\end{center}
	If $ R $ is a regular local G-ring, then we have:
	\begin{enumerate}
		\item 	Let $ f \in M $, $ f \neq 0 $, be such that $ R / \langle f \rangle $ is a domain. 
		Then $ \widehat{R} \,/\,\langle f \rangle\hspace{-2pt}\cdot\hspace{-3pt}\widehat R $ is reduced.
		\item Any quotient of $ R $, localization $ R_P $ at a prime ideal $ P \subset R $, or localization of a polynomial ring $ R[T] $ at a maximal ideal is again a G-ring.
	\end{enumerate}
\end{Lem}

We assume that the following claim holds true:

\begin{Hypothesis}[Empty Case]
	\label{Hyp:EmptyCase}
	Let $ \cR $ be a regular local G-ring and $ I \subset \cR $ a non-zero ideal.
	Let $ ( t , x ) = ( t_1, \ldots, t_d; x_1, \ldots, x_s ) $ be a regular system of parameters for $  \cR $  such that Hypothesis~\ref{hyp} is true
	%
	for
	$ ( \cR, (t, x ), I ) $. 
	Let $ (P) = (P_1, \ldots, P_m ) $ be a $ 0 $-normalized $ ( t ) $-standard basis for $ I $
	and set $ \nu_i := n_{(u)} (P_i) $, for $ 1 \leq i \leq m $. 
	
	We assume  $ \cpoly It = \varnothing $.	%
	By Hironaka's Theorem \ref{Thm:Hironaka}, there exist elements
	$ ( \hx ) = ( \hx_1, \ldots, \hx_s ) $ such that 
	$ ( t , \hx ) $ is a regular system of parameters for $ \widehat{ \cR} $, 
	$ (\hx ) $ determines the directrix of $ I' := I \cdot \cR/\langle t \rangle  $, 
	and $ \poly It\hx = \varnothing $. 
\end{Hypothesis}

	\begin{Claim}\label{claim:polyvide}
	Under Hypothesis~\ref{Hyp:EmptyCase}, there exist 
	a $ 0 $-normalized $ ( t ) $-standard basis $ ( Q ) = (Q_1, \ldots, Q_m ) $ of $ I $ in $ \cR $ and
	elements $ ( z ) = ( z_1, \ldots, z_s ) $ in $  \cR $ such that 
	\begin{enumerate}
		\item 	
		$ ( t, z) $ is a regular system of parameters for $  \cR $,
		
		\item 
		$ ( z ) $ yields the directrix of $ I' $,
		
		\item 
		$ \langle z \rangle \cdot \widehat{ \cR} = \langle \hx \rangle $, 
		
		\item $ z_j \in I(  \poly Ptx;1)  $, for $ 1\leq j \leq s $,
		
		\item 
		$ Q_1 = P_1 $ and
		$ Q_i = P_i + \sum\limits_{a=1}^{i-1}  H_{i,a} P_a  $,
		for $ H_{i,a} \in I( \poly Ptx ; \nu_{i}-\nu_a) $, for all $ i \in \{ 2, \ldots, m \} $,
		and
		
		\item 
		$ 	\poly Qtz = \poly Itz = \cpoly It = \varnothing . $
	\end{enumerate} 
\end{Claim}

\noindent 
(We switched the notations slightly since, most of the time, we apply the hypothesis
 and the claim in case of graded rings associated to $ J \subset R $).

In the case of an empty polyhedron, a $ 0 $-normalized set of generators for $ (g) $ such that $ \poly gtz = \varnothing $ is also vertex-normalized since the conditions in the definition are empty. 

\begin{Rk}\label{Rk:hyp*}
Unfortunately, this claim can be proven only with the supplementary hypothesis $ (*) $ or $ (\Pol) $ or $ R $ Henselian, it is an open and challenging question to prove it with less restrictive conditions.

	The setting in the  proposition below is more general than in 
	Theorems~\ref{MainThm:Hensel}, or~\ref{MainThm:*}, or~\ref{MainThm:Pol} 
	since we do neither require $ R $ to be Henselian nor $ (*) $ nor $ (\Pol) $ to hold in Hypothesis~\ref{Hyp:EmptyCase} or in Claim~\ref{claim:polyvide}: 
	both authors hope that hypotheses will be skipped in the future.
	
	In Proposition~\ref{Prop:*_and_Pol_stable}, we prove the stability of $ (*) $ and $ (\Pol) $
	and discuss the Henselian case
	in order to conclude the proof of the main theorems modulo the Claim~\ref{claim:polyvide}.
\end{Rk}

\begin{Thm}
	\label{Thm:ReducEmpty}
	Suppose Claim~\ref{claim:polyvide} is true.
	Let $ R $ be a regular local G-ring, $ J \subset R $ a non-zero ideal and $ ( u , y ) = ( u_1, \ldots, u_e; y_1, \ldots, y_r ) $ a regular system of parameters~of $ R $  such that
	 such that Hypothesis~\ref{hyp} is true for $ ( R, (u, y ), J ) $.
	Assume that $ \cpoly Ju \neq \varnothing $.
	
	There exist a vertex-normalized $ ( u ) $-standard basis $ ( g ) = (g_1, \ldots, g_m ) $ of $ J $ and elements $ ( z ) = ( z_1, \ldots, z_r ) $ in $ R $ such that $ ( u, z) $ is a regular system of parameters for $ R $, $ ( z ) $ yields the directrix of $ J' = J \cdot R / \langle u \rangle $, and
	\[   
		\poly guz = \poly Juz =  \cpoly Ju .
	\]
\end{Thm}

\smallskip

For the proof, we use the analogous measure for the difference between 
$  \poly fuy  $ and $ \cpoly Ju $ as it is used in the proof of Theorem 3.3 in \cite{CPcompl}.

\begin{Def}
	\label{Def:measure}
	\begin{enumerate}
		\item 	Let $ L \in \IL_0 ( \IR^e ) $ be any semi-positive linear forms on $ \IR^e $.11
		For a non-empty subset $ \Delta \subset \IR^e_{\geq 0 } $ we set
		\[ 
		\delta_L ( \Delta ) := \min \{ L( v ) \mid v \in \Delta  \}	< \infty .
		\] 
		\item	Let $ \varnothing \neq \Delta^0, \Delta^+ \subset \IR_{\geq 0 }^e $ be two non-empty rational polyhedra, where one is contained in the other, $ \Delta^+ \supset \Delta^0 $.
		Let $ L_1, \ldots, L_n \in \IL_0 ( \IR^e ) $ be rational semi-positive linear forms defining the faces of $ \Delta^0 $,
		$
		\Delta^0 = \bigcap\limits_{ j = 1 }^n \Delta ( L_j ).
		$
		We set, for every $ j \in \{ 1, \ldots, n \} $,
		\[ 
		\ell_j ( \Delta^+ ) := \delta_{L_j} ( \Delta^+ ).
		\] 
		By construction, $ 0 \leq \ell_j (  \Delta^+ ) \leq 1 $. 
		If  $ \ell_j (  \Delta^+ ) = 1 $, then the face of $ \Delta^0 $ defined by $ L_j $ is contained in the face of $ \Delta $ defined by $ L_j $.
		The measure for the total difference is the non-negative rational number
		\[
		\Lambda_{\Delta^0} ( \Delta^+ ) := \sum_{ j = 1 }^n \Big( 1 - \ell_j ( \Delta^+ ) \Big) \in \frac{1}{ \beta ! \, \alpha!} \, \IZ_{\geq 0},
		\] 
		where $ \beta $ denotes the biggest denominator appearing in the coordinates of the (finitely many) vertices of $ \Delta^+ $ and $ \alpha $ is the biggest denominator appearing in the coefficients of $ L_1, \ldots, L_n $. 
		\item 
		Suppose $ \Delta^0 = \cpoly Ju \neq \varnothing $ and $ \Delta^+ = \poly fuy $, for some $ 0 $-normalized $ ( u ) $-standard basis $ (f) = (f_1, \ldots, f_m ) $.
		Then, we only write
		\[ 
		\hspace{25pt}
		\ell_j ( f, u, y ) := \ell_j (  \poly fuy ),
		\ \  \
		\Lambda ( f, u, y ) := \Lambda_{\cpoly Ju} (  \poly fuy ) 
		\]
	\end{enumerate}
\end{Def}

As in \cite{CPcompl}, we follow Hironaka \cite{HiroCharPoly} (2.6) and use in our proof of Theorem~\ref{Thm:ReducEmpty}.
the ideal of initial forms associated to a linear form as in the previous definition. 
The latter arises from the following valuation that is very useful for the study of the polyhedron
and which was already used in \cite{CPcompl}.

\smallskip

We fix:
Let $ R $ be a regular local ring and $ J \subset R $ be a non-zero ideal.
	Let $ ( u, y) = (u_1, \ldots, u_e ; y_1, \ldots, y_r ) $ be a regular system of parameters for $ R $
	such that $ (y) $ determines the directrix of $ J' $.
	Let $ (f) = (f_1, \ldots, f_m) $ be a $ ( u ) $-standard basis for $ J $.
	Hence, $ \poly fuy \subset \IR^e_{\geq 0} $ is defined (Definition~\ref{Def:asso_poly}).  
	Further, let $ L : \IR^e \to \IR $ be any non-zero rational semi-positive linear form, say defined by 
	$ L(v) = a_1 v_1 + \ldots + a_e v_e $, for $ v = (v_1, \ldots, v_e) \in \IR^e $ and $ a_1,\ldots, a_e \in \IQ_{\geq 0} $.
	We set $ \ell (f,u,y) := \delta_L(\poly fuy) $ 
	
\begin{Def}
	\label{Def:val_Luyf}
	We introduce the monomial ideals
	\[
		I_\lambda := \langle u^A y^B \mid L(A) + \ell ( f, u , y ) \cdot |B| \geq \lambda \rangle \subset R,
		\ \ \
		\mbox{ for } \lambda \geq 0. 
	\]
	This provides a monomial valuation $ \nu := \nu_{ L, u, y, f }  $ on $ R $:
	for $ g \in R $, $ g \neq 0 $, we define
	\[
		\nu ( g ) := {\sup} \{ \lambda \in \IR \mid g \in I_\lambda \}
	\]
	and $ \nu(0) := \infty $.
	The {\em graded ring associated to $ \nu $} is defined as
	\[
	\gr_\nu (R) := 	\gr_{\nu_{L,u,y,f} (R)} := 
	\bigoplus_{ \lambda \in \IR_{\geq 0 } }
	\cP_\lambda / \cP_\lambda^+, 
	\]
	where $ \cP_\lambda := \{ g \in R \mid  \nu(g) \geq \lambda \} \subset R  $ and 
	$ \cP_\lambda^+ := \{ g \in R \mid  \nu(g)  > \lambda \} \subset R $.

	For $ g \in R, g \neq 0 $, the {\em initial form of $ g $ \wrt $ \nu $}, denoted by 
	\[  
		\ini_{ \nu } (g) = \ini_{ \nu_{L,u,y,f} } (g) 
		\in \cP_{\nu(g)} / \cP_{\nu(g)}^+  \subset \gr_{ \nu } ( R), 
	\]
	is defined as the image of $ g $ under the canonical projection $ \cP_{\nu(g)} \to  \cP_{\nu(g)} / \cP_{\nu(g)}^+ $.
	We set $ \ini_{ \nu } (0) := \ini_{ \nu_{L,u,y,f} } (0)  := 0  $
	and $ \gr_\nu (R)_\lambda := \cP_{\lambda} / \cP_{\lambda}^+ $, for $ \lambda \geq 0 $.
\end{Def}

If we have a finite expansion $ g = \sum\limits_{(A,B)} C_{A,B} u^A y^B $, for $ g \in R \setminus \{ 0 \} $, then
\[  
\nu_{ {L,u,y,f} } (g) = \inf \{ L(A) + \ell(f,u,y) |B| \mid C_{A,B} \neq 0 \}. 
\]

\smallskip

\noindent 
{\bf Warning:} Do not confuse $ \ini_{\nu} ( g) $ with the initial form at a vertex, $ \ini_v ( g ) $, 
Definition \ref{Def:nlzd_in_v}(2).
Also, let us point out that $ \nu_{L,u,y,f}(g) $ is different from 
$ v_{L} (g)_{(u,y)} = \inf \{ L(A) + |B| \mid C_{A,B} \neq 0 \} $ (Definition~\ref{Def:with_L}(2)).
Indeed, if $ \ell(f,u,y) = 0 $, then $  \nu_{L,u,y,f}(y_j) = 0 $, while $ v_L(y_j)_{(u,y)} = 1 $, for $ 1 \leq j  \leq r $.
On the other hand, if $ \ell := \ell(f;u;y) \neq 0 $, we may introduce the semi-positive linear form $ \widetilde L : \IR^e \to \IR $ by $ \widetilde L (v) := \frac{1}{\ell} L (v) $. 
Then $ \widetilde\ell(f;u;y) := \delta_{\widetilde L}(\poly fuy) =1 $ and $ v_{L} (g)_{(u,y)} = \nu_{\widetilde L,u,y,f}(g) = \frac{1}{\ell} \cdot \nu_{L,u,y,f}(g) $, for $ g \in R $.

\medskip

For a given semi-positive linear form $ L $ on $ \IR^e $, as above,
let 
\[ 
\begin{array}{lcl} 
\cI &:=& \{ i \in \{ 1, \ldots, e \} \mid  a_{i} \neq 0 \} ,
\\
\cI' &:=& \{ i \in \{ 1, \ldots, e \} \mid  a_{i} = 0 \}  = \{ 1, \ldots , e \} \setminus \cI.
\end{array} 
\]
Further, let us introduce the notation
\[
\begin{array}{llllll}
U_i := \ini_{ \nu }(u_i),
& 
\mbox{for } i \in \cI,
& \ \ &
Y_j := \ini_{ \nu }(y_j),
& 
\mbox{for } 
1 \leq j \leq r.
\end{array} 
\]
Note that we have $ \ini_\nu (u_i) \in \gr_\nu (R)_0 $, for $ i \in \cI' $.
Moreover, if $ \ell(f,u,y) = 0 $ then we also have $ Y_j = \ini_\nu (y_j)\in \gr_\nu (R)_0 $, for $ 1 \leq j \leq r $.

\begin{Prop}[\cite{CPcompl} Proposition 2.3]
	\label{Prop:gr_nu_luyf}
	The graded ring $ \gr_{\nu_{L,u,y,f}}(R) $ of $ R $ \wrt $ \nu_{L,u,y,f} $ is given by:
	\begin{enumerate}
		\item 
		If $ \ell(f,u,y) \neq  0 $, then
		$ 
			\gr_{\nu_{L,u,y,f}} (R) = \dfrac{R}{\langle \{u_i\}_{i\in \cI}, y\rangle} [\{ U_i \}_{i \in \cI}, Y_1, \ldots, Y_r];
		$ 
		
		\medskip 
		
		\item 
		If $ \ell(f,u,y) = 0 $, then
		$ 
		\gr_{\nu_{L,u,y,f}} (R) = \dfrac{R}{\langle \{u_i\}_{i\in \cI} \rangle} [\{ U_i \}_{i \in \cI} ];
		$ 
	\end{enumerate} 
	In particular, we have 
	$ \gr_{\nu_{L,u,y,f}}(R) = k[U_1, \ldots, U_e, Y_1, \ldots, Y_r] $ whenever $ L $ is positive,
	where $ k $ denotes the residue field of $ R $.
\end{Prop}

With this preparation we can give the

\begin{proof}[Proof of Theorem~\ref{Thm:ReducEmpty} (modulo Claim~\ref{claim:polyvide})]
	We follow the proof of \cite{CPcompl} Theorem~3.3.
	Let  $ ( f ) = ( f_1, \ldots, f_m ) $ be a $ 0 $-normalized $ ( u ) $-standard basis of $ J $.
	Set $ \nu_i := n_{(u)} (f_i) $, for $ 1 \leq i \leq m $.
	Hironaka's preparation process (Theorem~\ref{Thm:Hironaka}) provides $ (\hy) := ( \hy_1, \ldots, \hy_r) $ and $ (\hf ) = ( \hf_1, \ldots, \hf_m ) $ in $ \hR $ such that $ ( u, \hy ) $ is a regular system of parameters for $ \hR $, $ ( \hy ) $ yields $ \Dir ( J' ) $, and
	$
	\poly{ \hf }{ u }{ \hy } = \cpoly{ J }{ u }.
	$
	
	Consider
	$ 
		\Lambda ( f, u, y )= \sum\limits_{ j = 1 }^n \big( \, 1 - \ell_j (f ,u , y) \, \big) \geq 0,
	$ 
	the measure introduced in Definition~\ref{Def:measure}, where $ L_1, \ldots, L_n $ are semi-positive linear forms on $ \IR^e $ defining $ \cpoly Ju $.
	Let $ a_{  j,i } \in \IQ_{\geq 0} $ be the rational numbers defining $ L_j $.
	We have that 
	
	\smallskip 
	
	\begin{itemize}
		\item 
		$ L_j (v_1, \ldots, v_e ) =  a_{j,1} v_1 + \ldots + a_{j,e} v_e $, for $ (v_1, \ldots, v_e) \in \IR^e $, $ 1 \leq j \leq n $.

\smallskip 
		
		\item 
		$ \cpoly Ju = \{ v = (v_1, \ldots, v_e ) \in \IR^e_{\geq 0 } \mid L_j (v) \geq 1, \mbox{ for } 1 \leq j \leq n \} $,
		
	\smallskip 
		
		\item 
		$ L_j(\cpoly Ju) = [1, +\infty [ $, for all $ j \in \{ 1 , \ldots, n \} $.
	\end{itemize}

\smallskip 

	When $ \Lambda ( f,u,y) = 0 $, we have $ \poly fuy = \cpoly Ju $ and we take $ (z) = (y) $.
	Since $ \Lambda ( f,u,y) = 0 $, the vertex normalization process cannot eliminate any of the vertices. 
	As there are only finitely many vertices, we obtain the desired vertex-normalized $ ( u ) $-standard basis $ ( g ) $ from $ ( f ) $ using Proposition~\ref{Prop:Nlz}.

	Suppose $ \Lambda ( f,u,y) > 0 $.
	Then there exists at least one $j \in \{ 1, \ldots, n \} $ such that 
	\[  
		L_j(\poly fuy) \supsetneq  [1, +\infty [
		\ \ \mbox{ or, equivalently,}
		\ \
		\ell_j (f ,u , y) < 1 .
	\]
	This means that the face of $ \poly fuy $ defined by $ L_j $ is solvable by Hironaka's process.
	We fix such a $ j $.
	Let us denote by $ \Delta_j $ the mentioned face,
	\[
		\Delta_j = \{ v \in \IR^e \mid L_j(v) = \ell_j (f,u,y) \} \cap \poly fuy .
	\]
	As in Hironaka \cite{HiroCharPoly} (2.6), 
	we consider the ideal of initial forms 
	\[
		\ini_{ \nu_j } ( J ) := \langle  \ini_{ \nu_j} ( g ) \mid g \in J \rangle \subset \gr_{ \nu_j }(R),
	\] 
	where $ \nu_j := v_{ L_j, u, y, f } $ is the valuation defined in Definition~\ref{Def:val_Luyf}.
	Since we fixed $ j $ with this property, let us set
	\begin{center}
		$ L := L_j $,
		\hspace{10pt}
		$ \nu := \nu_j $,
		\hspace{10pt}
		$ I_\nu :=  \ini_{ \nu } ( J ) $,
		\hspace{10pt} 
		and
		\hspace{10pt}
		$ \ell ( f ,u, y ) := \ell_j ( f, u , y ) $. 
	\end{center}
	We define 
	$ 
		t_i := \ini_\nu (u_i), 
	$ 
	for $ 1 \leq i \leq e $,
	and introduce 
	\[ 
		 \cR := \gr_{ \nu } ( R ) _{ \langle \, t_1, \ldots, t_e, Y_1, \ldots, Y_r  \, \rangle  }.
	\]
	Then $  \cR $ is a regular local ring with regular system of parameters~$ ( t, Y ) = ( t_1, \ldots, t_e, Y_1, \ldots, Y_r ) $ and residue field $  \cR / \cM = k = R / M $, where we denote by $ \cM  := \langle t, Y \rangle $ the maximal ideal of $  \cR $.
	We set 
	\[
	I := I_\nu \cdot  \cR \subset  \cR .
	\] 
	By Lemma \ref{lem:G=>excellent}, $  \cR $ is a G-ring.
	%
	The graded structure of $ \gr_{ \nu } ( R ) $ induces a monomial valuation on $  \cR $, again denoted by $ \nu $.
	Moreover, $ \nu $ extends canonically to the $ \cM  $-adic completion $ \widehat  \cR $ of $  \cR $. 
	We have inclusions $ \cR \subseteq \gr_{ \nu } ( \hR )_{ \langle t, Y \rangle } \subseteq \widehat \cR $ and an isomorphism
	{(use Proposition~\ref{Prop:gr_nu_luyf})}
	\[
	\widehat \cR \; \cong \; 
	\left\{
	\begin{array}{cc}
	\, \gr_{ \nu } ( \hR)_0 [[ Y, \{ U_i \}_{i \in \cI }]], 	& \ \mbox{if } \ell ( f ,u, y ) \neq 0,  \\[5pt] 
	\, \gr_{ \nu } ( \hR)_0 [[ \{ U_i \}_{i \in \cI }]], 		& \ \mbox{if }  \ell ( f ,u, y ) = 0. 
	\end{array} 
	\right.
	\] 
	
	\medskip
	
	We define 
	\[ \
		P_1 := \ini_{ \nu } ( f_1 ) , \ldots, P_m :=  \ini_{ \nu } ( f_m ) \in \cR.
	\]
	
	\begin{Claim}
		\label{Claim:P_0_nlz_u_std}
		 $ ( P ) = (P_1, \ldots, P_m ) $ is a $ 0 $-normalized $ ( t )$-standard basis of $ I $.
	\end{Claim}
	
\begin{proof}
	(i) Case $ \ell ( f ,u, y )\not =0$. Let $\Phi \in  \gr_{ \nu } ( R )_\mu \cap I_\nu$, let $g\in J$ be such that $ \ini_\nu(g)=\Phi$, let $\Delta:=\poly fuy$. To simplify notations, we replace $L$ by $ {1 \over \ell ( f ,u, y )}L$, then $ \ell ( f ,u, y )=1$. 
	We have $ g \in I(\Delta,\mu) $ (Definition~\ref{Def:I(Delta,b)}).
	By \cite{HiroCharPoly}~Theorem~(2.21), we can write $ g $ as
	\[
		g=\sum_{i=1}^m h_i f_i+g_+,\ 
		\mbox{ for } h_i\in I(\Delta;\mu-\nu_i), \ 1\leq i \leq m, 
		\mbox{ and } \nu(g_+)>\mu.
	\]
	As $\{ x \in \IR_{\geq 0}^e \mid L(x)=1 \} \cap \Delta $ is a face of $\Delta $,
	we have $ \nu(h_i)\geq \mu-\nu_i$.
	So, 
	the class of $ h_i $ in $ \gr_{ \nu } ( R )_{\mu-\nu_i} $ is defined,
		\[
			\cl_{\mu-\nu_i} (h_i) := \left\{
			\begin{array}{ll}
			\ini_\nu (h_i),	&
			\mbox{if } \nu(h_i) = \mu-\nu_i,
			\\
			0,	&
			\mbox{if } \nu(h_i) > \mu-\nu_i,
			\end{array}
			\right. 
		\]
	and	
	$\Phi=  \sum\limits_{i=1}^m H_i P_i,\ \mathrm{with}\ H_i:= \cl_{\mu-\nu_i}(h_i)\in \gr_{ \nu } ( R )_{\mu-\nu_i}$, $1\leq i \leq m $. The end is left to the reader.
	
	\smallskip 
	
	\noindent
	(ii) Case $ \ell ( f ,u, y ) =0$.  Let $\Phi \in  \gr_{ \nu } ( R )_\mu \cap I_\nu$, let $g\in J$ be such that in$_\nu(g)=\Phi$, let $\Delta:= \poly fuy $.
	Furthermore, let $b\in \IQ_+$ be the largest rational number 
	such that $g\in I(\Delta;b)$ and 
	$N:=\{ x\in \IZ_{\geq 0}^e \mid  L(x)\geq  \mu \}$. 
	We have $ g \in I(\Delta ,b) \cap I(N) $,
		where $ I(N) := I(N)_u := \langle u^A \mid A \in N \rangle \subset R $.
	Then, \cite{HiroCharPoly}~Theorem~(2.21) gives:
	\[ 
	g=\sum_{i=1}^m h_i f_i+g_+,
	\ 
	\mbox{ for } h_i\in I(\Delta; b -\nu_i) \cap I(N) \subset  R, 
	\ 
	1\leq i \leq m, 
	\]
	\[
	\mbox{and  }
	\nu	(g_+)\geq \mu
	\mbox{ with } \ g_+ \in I(\Delta,b_+),\ 
	\mbox{ for some } b_+>b.
	\]
	As $\bigcap\limits_{ c \in \IQ_+}I(\Delta;c)=\varnothing$, by faithful flatness, we can get $g_+=0$.
	Since we have $ h_i \in  I(N) $, we get $ \nu(h_i) \geq \mu$, 
	while $\nu(f_i)=0$, for $1\leq i \leq m$.
	Therefore, we obtain
	$\Phi=  \sum\limits_{i=1}^m H_i P_i,\ \mathrm{with}\ H_i= \cl_{\mu}(h_i)\in \gr_{ \nu } ( R )_{\mu}$, for  $1\leq i \leq m $.  
	(Here, the class of $ h_i $ is defined analogously to above).
	The end is left to the reader.
\end{proof}	
	
	\noindent 
	{\em Continuing the proof of Theorem~\ref{Thm:ReducEmpty}.}
	We define 
	\[ 
		\hY_1 := \ini_{ \nu } ( \hy_1 ), \ldots, \hY_r := \ini_{ \nu } ( \hy_r )  \in \gr_{ \nu } ( \hR ),
	\ \ \mbox{ and}
	\] 
	\[
		\hP_1 := \ini_{ \nu } ( \hf_1 ) , \ldots, \hP_m :=  \ini_{ \nu } ( \hf_m ) \in \widehat\cR. 	
	\]
	By Hironaka's construction,
	$ \hY_j \in \gr_{ \nu } ( \hR )_{ \ell_j (f,u,y) } $, for every $ j \in \{ 1, \ldots, r \} $, 
	and
	$ ( t, \hY ) $ is a regular system of parameters for $ \widehat \cR $, 
	and $ \poly \hf u\hy \subsetneq \poly fuy $.
	
	By the choice of the linear form $ L $ and since $ \poly \hf u \hy = \cpoly Ju $, we have 
		$
			\ell( \hf,u,\hy) = 1 > \ell (f,u,y).
		$
		If we consider expansions of $ \hf_1, \ldots, \hf_r $ and use the definition of $ \ell( \hf,u,\hy) $,
	we obtain $ \poly \hP t\hY = \varnothing $ which implies $ \cpoly It = \varnothing $.
	Furthermore, $ I \subset \cM  $,
thus we can apply Hypothesis~\ref{Hyp:EmptyCase} 
	and Claim~\ref{claim:polyvide}.  
	Therefore, there exist $ Z_1, \ldots, Z_r \in \cM  $ 
	and a $ 0 $-normalized $ (t) $-standard basis $ (Q) = (Q_1, \ldots, Q_m ) $ such that 
	$ Q_i = P_i + \sum\limits_{a=1}^{i-1}  H_{i,a} P_a \in \cR $,
	for $ H_{i,a} \in I( \poly Ptx ; \nu_{i}-\nu_a) $,
	and
	$ (t, Z) $ is a regular system of parameters for $ \cR $, 
	$ \langle Z \rangle \cdot \widehat \cR = \langle \hY \rangle \subset \widehat \cR $, and
	\begin{equation} 
	\label{eq:in_nu_poly_empty}
	\poly{ Q }{ t }{ Z } = \varnothing.
	\end{equation}

	Since $ Z_j \in I(\Delta(P,t,Y);1)  $ and 	
	$ \langle Z \rangle \cdot \widehat \cR = \langle \hY \rangle $,
	we are able to choose lifts,
	$ z_j = y_j - \vp_j \in R $ of $ Z_j $ in $ R $ such that $ z_j \in  I(\Delta(f,u,y);1) $, for $ 1 \leq j \leq r $. 
	Analogously, 
	we may assume without loss of generality that $ H_{i,a} \in \gr_\nu (R) $.
	We have that $ H_{i,a} $ are homogeneous (of degree $ \ell(f,u,y) (\nu_i - \nu_a) $ \wrt $ \nu $).
	So, we can choose lifts 
	$ h_{i,a} \in I( \poly fuy ; \nu_{i}-\nu_a ) \in R $, for $ 1 \leq a \leq i \leq m $,
	and we define $ (g) = (g_1, \ldots, g_m) $ by putting 
	\[ g_i := f_i + \sum\limits_{a=1}^{i-1}  h_{i,a} f_a, 
	\ \ \ \mbox{ for } \ \
	1 \leq i \leq m . 
	\]   

	Let $ \cL  $ be a positive linear form on $ \IR^e $
such that	
	\[
		\delta_\cL (f,u,y) = \inf \{ \cL(v) \mid v \in \poly fuy \} > 1.
	\]
	Recall that $ v_\cL ( \ ) $ denotes the monomial valuation defined via $ v_\cL(u^Ay^B)_{(u,y)} = \cL(A) + |B| $
	(Definition~\ref{Def:with_L}(2)).
	In particular, we have $ v_\cL(z_j)_{(u,y)} = v_\cL(y_j)_{(u,y)} $ for every $ j \in \{ 1, \ldots, r\} $.
	Hence, $ v_\cL(g)_{(u,z)} = v_\cL(g)_{(u,y)} $, for every $ g \in R $,
	and we can drop the reference to the regular system of parameters. 
	Since $ h_{i,a} \in I( \poly fuy ; \nu_{i}-\nu_a ) $, we have 
	\[
		v_\cL(g_i) =  v_\cL\Big(f_i + \sum\limits_{a=1}^{i-1}  h_{i,a} f_a\Big) \geq \nu_i
		\ \ 
		\mbox{ for } 1 \leq i \leq m.
	\]
	Since $ (f) $ is $ 0 $-normalized, we must have equality,
	and, furthermore, we have 
	$ \langle \ini_\cL (g_1), \ldots, \ini_\cL(g_m) \rangle  = \ini_\cL (J) $.
	Hence, $ (g_1, \ldots, g_m )$ is a $ (u)$-effective standard basis of $ J $ 
	with reference datum $ (z,\cL ) $. 
	This provides
	\[
		\delta_\cL(g,u,z) \geq \delta_\cL(f,u,y)
	\]
	and as this inequality holds for every positive $ \cL $ as above, we have an inclusion 
	$ \poly guz \subseteq \poly fuy $.

	On the other hand, \eqref{eq:in_nu_poly_empty} is equivalent ot $ \ell(g,u,z) > \ell(f,u,y) $ for our fixed linear form $ L $.
	We conclude that
	\[ 
		0  \leq \Lambda ( g, u, z ) < \Lambda ( f, u, y ). 
	\]
	Since $ \Lambda (.) $ takes only values in a discrete subset of $ \IQ_{\geq 0 } $, 
	the strict decrease can only happen finitely many times. 
	We repeat the arguments of this proof for $ \poly guz $ instead of $ \poly fuy $. 
	Eventually, we obtain in finitely many steps a vertex-normalized $ (u) $-standard basis $ ( g^* ) $ of $ J $ and elements $ ( z^* ) $ in $ R $ extending $ ( u ) $ to a regular system of parameters for $ R $ with all desired properties. 
	In particular, 
	$ \poly{g^*}u{z^*} = \cpoly Ju. $
\end{proof}

\smallskip 

Using the notation of the proof, we have 

\begin{Lem} 
	\label{Lem:tangent_iso}
	The tangent cone $ C_t (I) \subset \IA_k^r $ 
	of $ I' := I \cdot \cR/\langle t \rangle $ \wrt $ \cM' := \cM \cdot \cR/\langle t \rangle $
	is isomorphic to the tangent cone 
	$ C_u (J)  \subset \IA_k^r $ of $ J' $ \wrt $ M' = M \cdot R / \langle u \rangle $. 
	In particular, their directrices (resp.~ridges) are isomorphic.
	Hence, the following equalities hold:
	\[  
	\begin{array}{c} 
		\dim \Dir (C_{t}(I)) =\dim \Dir (C_{u}(J)) \\
		\dim \Rid (C_{t}(I)) = \dim \Rid (C_{u}(J))
	\end{array} 
	\] 
	Moreover, we have $ \HS(C_{t}(I)) \cong \HS(C_{u}(J)) $.
\end{Lem} 

\begin{proof}
	Let $ g \in R $ and put $ G := \ini_\nu(g) $.
	If $ \overline{G} $ denotes the image of $ G $ in $ \cR/ \langle t \rangle $, 
	then we observe that $ \ini_{\cM'}( \overline{G} ) $ coincides with $ \ini_{M'}(\overline{g}) $,
	where $ \overline{g} = g \mod \langle u \rangle $
	(up to renaming the variables).
	Therefore, $ C_t (I) \subset \IA_k^r $ and $ C_u (J)  \subset \IA_k^r $ are isomorphic. 
	The remaining part is clear. 
\end{proof}

In general, without assumptions $ (*) $, $ (\Pol) $, or $ R $ Henselian, we have

\begin{Prop}
	\label{Prop:Pol_often_and_else_smaller_dim}
	Let $ R $ be a regular local $ G $-ring
	and
	let $ J \subset R $ be a non-zero ideal.
	Let $ (u,y) $ be a regular system of parameters for $ R $
	such that $ (u) $ is a regular $ (R/J) $-sequence
	and $ (y) $ determines the directrix of $ J' = J \cdot R/\langle u \rangle $. 
	Let $ (f) = (f_1, \ldots, f_m) $ be a $ 0 $-normalized $ (u) $-standard basis for $ J $. 
	
	Using the notation of the proof of Theorem~\ref{Thm:ReducEmpty}, we have:
	\begin{enumerate}
		\item 
	If $ \ell(f,u,y) > 0 $,
	then hypothesis $ (\Pol) $ holds for $ (I,\cR,\cS, P, t,Y) $,
	where 
	\[
		\cS := 
		{R \over \langle \{u_i\}_{i\in \cI}, y \rangle} [\{ U_i\}_{i\in \cI} ]_{ \langle t \rangle },
	\]
	
		\item 
		If $ \ell(f,u,y) = 0 $,
		then
		the important data of $ (I,\cR, P, t,Y) $ is contained in $ \cR_0 := \gr_{\nu} (R)_0  = R/\langle \{ u_i\}_{i\in \cI} \rangle $ (Proposition~\ref{Prop:gr_nu_luyf})
		and we have $ \dim ( \cR_0 ) < \dim (R) $. 
		\\
		More precisely,
			$ P_i = \ini_{ \nu }(f_i) , Y_j = \ini_{ \nu }(y_j) \in  \gr_{\nu} (R)_0 $,
			for all $i , j $. 
	\end{enumerate}
\end{Prop}

\begin{proof}
	Part (2) is an easy observation. 
	
	For (1), set $ S :=  {R}/{\langle \{u_i\}_{i\in \cI}, y\rangle} $.
	Recall that $ \cR $ is the localization of $ \gr_{\nu} (R) = S [\{ U_i \}_{i \in \cI}, Y_1, \ldots, Y_r] $  (Proposition~\ref{Prop:gr_nu_luyf}) 
	at the maximal ideal $ \langle t , Y \rangle $. 
	Furthermore, for $ 1 \leq i \leq m $, we have 
	$ P_i = \ini_\nu (f_i) \in \gr_{\nu} (R) \subset \cS [Y] $
	and we have seen in Claim~\ref{Claim:P_0_nlz_u_std} that $ (P_1, \ldots, P_m ) $ is a $ 0 $-normalized $ (t)$-standard basis of $ I $. 
	Since $ \nu(y) = \ell(f,u,y) > 0 $ (Definition~\ref{Def:val_Luyf}),
	we have $ \deg_Y(P_i) = n_{(u)} (P_i ) $. 
	Finally, Lemma~\ref{Lem:tangent_iso} implies that $ (Y) $ determines the directrix of $ I' $,
	since $ (y) $ determines the directrix of $ J'  $.  
	In total, we obtain that
	condition $ ( \Pol ) $ holds for $ (I,\cR,S, P, t,Y) $.
\end{proof}

\begin{Rk}
	As a consequence, if $ \ell(f,u,y) > 0 $, it is sufficient if we can handle Claim~\ref{claim:polyvide} under hypothesis $ (\Pol) $
	(which will follow from Theorems~\ref{Thm:Strongly_Nlz} and~\ref{Thm:Claim_2.3_pol})
	in order to apply the techniques of the proof of Theorem~\ref{Thm:ReducEmpty}
	to get closer to the characteristic polyhedron.
	
	On the other hand, if $ \ell(f,u,y) = 0 $, 
	an induction on the dimension of the ambient ring may be applied. 
	Yet, we still have to require $ (*) $ to hold for the original data $ J \subset R $  or $ R $ to be Henselian 
	in order to be able to prove Claim~\ref{claim:polyvide}.
\end{Rk}

\begin{Prop}
	\label{Prop:*_and_Pol_stable}
	Let $ (J,R,f,u,y) $ and  $ (I,\cR,P,t,Y) $ be as in the proof of Theorem~\ref{Thm:ReducEmpty}.
	\begin{enumerate}
		\item 
		Let $ (S,N,k) $ be a regular local $ G $-ring contained in $ (R,M,k) $.
		Suppose $ (u) $ is a regular system of parameters for $ S $.
		If hypothesis $ (\Pol) $ holds for $ (J,R,S,f,u,y) $, 
		then $ (\Pol) $ also holds for $  (I,\cR,\cS,P,t,Y) $,
		where 
		(using the notation of Proposition \ref{Prop:gr_nu_luyf})
		\[ 
		\cS:= 
		\left\{ 
		\begin{array}{ll}
		\displaystyle 
		{R \over \langle \{u_i\}_{i\in \cI}, y \rangle} [\{ U_i\}_{i\in \cI} ]_{ \langle t \rangle },
		& \mbox{ if } \ell(f,u,y) > 0 ,
		\\[15pt]
		\displaystyle {S \over \langle \{u_i\}_{i\in \cI} \rangle} [\{ U_i\}_{i\in \cI} ]_{\langle t \rangle },
		& \mbox{ if } \ell(f,u,y) = 0.
		\end{array} 
		\right. 
		\]
		\item 
		If $ (*)(a) $ (resp.~$(*)(b)$) holds for $ (J,R,u) $, then $ (*)(a) $ (resp.~$(*)(b)$) holds for $ (I, \cR, t) $, too.
		\item 
		If $ R $ is Henselian, then either 
		the important data of $ (I,\cR,P,t,Y) $ lies in $ \gr_\nu (R)_0 $, which is Henselian,  
		or $(\Pol) $ holds for $ (I,\cR,\cS, P,t,Y) $, for $ \cS $ as in part (1).
	\end{enumerate}
\end{Prop}

\begin{proof}
	Statement (1) an immediate consequence of Proposition~\ref{Prop:Pol_often_and_else_smaller_dim}.
	Its proof is left as an exercise to the reader.
	Further, if hypothesis $ (*)(a) $ holds for $ (J,R,u) $ then $ (*)(a) $ holds for $ (I,\cR,t) $, by Lemma~\ref{Lem:tangent_iso}. 

	Let us come to $ (*)(b) $.
	Suppose that $ \car(k) \geq \frac{\dim(X)}{2} + 1 $, where $ X = \Spec(R/J) $.
	Set $ \cX := \Spec (\cR/I) $.
	We claim that $ \dim (\cX) \leq \dim (X) $,
	from which we obtain immediately $ \car(k) \geq \frac{\dim(\cX)}{2} + 1 $.
	Every $ P_i = \ini_\nu(f_i)$ can be written as 
	$ 
	P_i = F_i(Y) + \sum\limits_{|B| < \nu_i } 
	 \lambda_{A,B,i} t^A Y^B,
	$
	with $ \lambda_{A,B,i} $ invertible or zero,
	and
	$
	\ini_\cM(F_i(Y)) = \ini_0(f_i)\in k[Y]$. 
	By taking roots of the parameters~$(t)$, 
	we may suppose that $ \ord_{\cM} (t^A y^B)> n_{(u)}(f_i) = n_{(t)}(P_i) $, for $1\leq i \leq m$, if $ \lambda_{A,B,i}$ is invertible.
	Then $(\ini_0(f_1),\ldots,\ini_0(f_m))\subset k[T,Y]$ (with some abuse of notation for $Y$) are in the ideal of the tangent cone of $\cX $ at the origin. 
	This gives the claimed inequality $ \dim(\cX) \leq \dim(X) $.  
	
	(3)
	Clearly, the property of being Henselian is not necessarily stable when passing to 
	$ \cR $, 
	which is the localization of a polynomial ring at a maximal ideal.
	Since being a Henselian $ G $-ring is stable by taking quotients, 
	$ \gr_\nu(R)_0 = R/\langle \{u_i\}_{i\in \cI} \rangle $ is also a Henselian regular local $ G $-ring, so Proposition~\ref{Prop:Pol_often_and_else_smaller_dim} yields the assertion.
\end{proof}

In conclusion,
we have reduced the problem of determining the characteristic polyhedron without passing to the completion 
to Claim~\ref{claim:polyvide},
the case of an empty characteristic polyhedron, $ \cpoly Ju = \varnothing $, 
where $ (u) $ is a regular $ (R/J) $-sequence.
While the reduction is valid without hypothesis $ (*) $ or $ (\Pol) $ or $ R $ Henselian, we need them to find suitable coordinates $ (z) = (z_1,\ldots, z_r) $.

%
%
%
%
%
%
%
%
%
%

\medskip

\section{Empty Characteristic Polyhedron I: Hypothesis $(\Pol)$}
\label{sec:empty(Pol)}

We begin our treatment of {the} case of an empty characteristic polyhedron with
the situation, where hypothesis $ ( \Pol ) $ holds.
Here, we can prove slightly stronger results. 
We introduce a stronger version of being normalized and show that it can also be achieved via a finite procedure.
Furthermore, we prove, once obtained, it is stable under translations in the variables $ (y) = (y_1, \dots, y_r) $. 
In the second subsection, we deduce that the desired coordinates $ (z_1, \ldots, z_r )$ can be obtained by a translation. 

Suppose that $ (\Pol) $ holds for $ (J,R,S,f,u,y) $.
Recall that this means:
\begin{itemize}
	\item 
	$ (S,N,k) $ is a regular local $G$-ring contained in $ (R,M,k) $, 
	with the same residue field, 
	and $ (u) = (u_1, \ldots, u_e) $ {is} a regular system of parameters for $ S $,
	\item 
	$ R = S[y_1,\ldots,y_r]_{M} $, $ J \subset R $ {is} a non-zero ideal such that the directrix of $ J' = J \cdot R/\langle u \rangle $ is determined by $ ( y ) $,
	and $ \cpoly Ju = \varnothing $,
	\item 
	$ (f) = (f_1, \ldots, f_m ) $ is a $ (u) $-standard basis of $ J $
	such that we have $ f_i\in S[y_1,\ldots,y_r] $ 
	with 
	$ \deg_{y}(f_i) = n_{(u)} (f_i) $,
	for $ 1 \leq i \leq m $.
	By Lemma~\ref{Lem:0Nlz}, we may suppose that $ (f) $ is $ 0 $-normalized and that the condition 
	$ \exp(f_i) \preceq \exp(f_{i+1}) $, for $ 1 \leq i < m $, 
	holds true, without loss of generality.
\end{itemize}

\begin{Rk} 
	\label{Rk:pol_important}
In Proposition~\ref{Prop:Pol_often_and_else_smaller_dim}(1), we have 
already seen
that $(\Pol)$ is an essential case
since it appears even in the general situation for $ R $ any regular 
local $ G $-ring (without restrictions on $ J \subset R $).

Let us note that in \cite{CP2} and \cite{CPmixed} starting from section 5, the case  $(\Pol)$ is 
automatically fulfilled and this helps the authors to control the 
behavior of the characteristic polyhedra after permissible blowing ups. 
In general, with the usual notations, blow up the origin and 
suppose there is a ``very near point" of parameters 
$(  u_1, v_2, \ldots, v_{e'}, \frac{y_1}{u_1},\ldots, \frac{y_r}{u_1} ) $ 
in the first chart
\[  
R \longrightarrow 
R \left[ u_1, \frac{u_2}{u_1},\ldots, \frac{u_e}{u_1}, \frac{y_1}{u_1},\ldots, \frac{y_r}{u_1} \right]_{\displaystyle \langle u_1,v_2,\ldots,v_{e'}, \frac{y}{u_1} \rangle } =:R',
\]
for suitable $(v_2, \ldots,v_{e'}) $.
Set $ v_1 := u_1 $ and $ y_j ' := \frac{y_j}{u_1} $, for $ 1 \leq j \leq r $.  
For every face of the polyhedron 
$\Delta(f';v;y')\subset \IR^{e'}$ of equation
$L (t_1, \ldots, t_{e'} ) =  a_{1} t_1 + \ldots + a_{e'} t_{e'}=  \ell (f',v,y')$
with $ a_1 > 0 $ positive, 
we have $\ell (f',v,y')>0$, 
i.e., we have $(\Pol)$ for 
$\gr_{\nu_{L,v,y',f'}} (R)$.
It appears that in \cite{CJS}, 
\cite{HomeworkDim2},
\cite{CP2}, and \cite{CPmixed}, the invariants are defined with these 
faces only.
\end{Rk}

%
%
%
%
%
%
%
%
%
%

\bigskip

\subsection{Strong normalization for $(\Pol) $}

First, we discuss a finite normalization process in order to obtain suitable generators for the ideal $ J \subset R $. 
Throughout the section, we assume that hypothesis $ (\Pol) $ holds for some given $ (J,R,S,f,u,y) $.

Given $ h \in R $, we have a finite expansion $ h = \sum C_{A,B} u^A y^B $ with coefficients $ C_{A,B} \in R^\times \cup \{ 0 \} $.
For the purpose of this subsection, the expansion does not provide sufficient control.
Since $ R = S[y]_M $, we may assume $ h \in S[y] $ after multiplying by a unit. 
Hence, we have
\[ 
\label{eq:expan_new}  	
h = \sum_{B: |B|\leq \nu} \vp_B \, y^B,
\ \ \
\mbox{ for } \vp_B =  \vp_B (u) \in S
\mbox{ and } \nu := \deg_{y}(h).
\]

\begin{Def} 
	\label{Def:strongly-nlzd}
	Let $ S $ be a regular local ring with regular system of parameters $ (u) = (u_1, \ldots, u_e ) $,
	let $ R = S[y_1, \ldots, y_r]_M $, for $ M = \langle u , y \rangle $.
	Let $ (f) = (f_1, \ldots, f_m) $ be a $ 0 $-normalized system of elements in $ S[y_1, \ldots, y_r] $ with $ \deg_{y}(f_i) = n_{(u)} (f_i) =: \nu_i $,
	for $ 1 \leq i \leq m $. 
	Let $ \{ \vp_{B,i} \in S \mid |B|\leq \nu_i \} $ be the coefficients appearing in the expansion of $ f_i $ of the form above, for $ i \in \{ 1, \ldots, m \} $.
	We say $ (f) =  (f_1, \ldots, f_m) $ is {\em strongly normalized
	(with respect to $ (y) $)}
	if we have $ \vp_{B,i} \equiv 0 $, for every $ i \in \{ 2, \ldots, m \} $
	and every $ B \in \exp( f_1, \ldots, f_{i-1} ) 
	=
	\exp ( \langle \ini_0 (f_1), \ldots, \ini_0( f_{i-1}) \rangle ) $
	(Definition~\ref{eq:def_exp(f1...fm)}). 
\end{Def} 
{\em Observe that we do not necessarily require $ S $ to be a $ G $-ring in this subsection.}
The main result here is

\begin{Thm}
	\label{Thm:Strongly_Nlz}
	Let $ S $ be a regular local ring with regular system of parameters $ (u) = (u_1, \ldots, u_e ) $,
	let $ R = S[y_1, \ldots, y_r]_M $, for $ M = \langle u , y \rangle $
	and $ J \subset R $ be a non-zero ideal.
	Let $ ( f ) = (f_1, \ldots, f_m) $ be a $ 0 $-normalized $ (u) $-standard basis for $ J $
	with $ f_i \in S[y] $ and $ deg_y(f_i) = n_{(u)}(f_i) =: \nu_i $,
	for $ 1 \leq i \leq m $.

	There exist $ x_{i,j} \in S[y] \cap \langle u \rangle $ with $ \deg_y(x_{i,j}) \leq \nu_i- \nu_j $, 
	for	$ 2 \leq i < j  \leq m $, such that,
	if we define 
	$ g_1 := f_1 $ and 
	$ g_i := f_i - \sum\limits_{j=1}^{i-1} x_{i,j} f_j $,
	for $ i \geq 2 $,
	then $ ( g) = (g_1, \ldots, g_m) $ is a $ (u) $-standard basis for $ J $ that is strongly normalized with respect to $ (y) $
	and $ \deg_y(g_i) = n_{(u)}(g_i) = \nu_i $.
\end{Thm}

{We do not have to modify $ f_1 $ and $ g_1 = f_1 $. Thus,}
we have to show the following (using the notions of the theorem):
Suppose $ (g_1, \ldots, g_{i-1}) $ are strongly normalized.
Then there is a finite process for constructing $ x_{i,j} $, with $ i < j $ such that $ (g_1, \ldots, g_{i-1}, g_i) $ are strongly normalized
(Proposition~\ref{Prop:m-1_to_m}).
After that, in Lemma \ref{Lem:remains(u)std}, we prove that $ (g) $ remains a $ (u) $-standard basis,
which completes the proof of Theorem~\ref{Thm:Strongly_Nlz}.
The property on the degree of $ g_i $ follows since $ (f) $ is $ 0 $-normalized and by the degree condition on $ x_{i,j} $.

\begin{Prop}
	\label{Prop:m-1_to_m}
	Let $ S $ be a regular local ring with regular system of parameters $ (u) = (u_1, \ldots, u_e ) $,
	let $ R = S[y_1, \ldots, y_r]_M $, for $ M = \langle u , y \rangle $.
	Let 
	$ (g_1, \ldots, g_{i-1}, f_i)  $ be a system of elements in $ S[y] $ that is $ 0 $-normalized
	and for which $ \deg_y (g_j) = n_{(u)}(g_j) =: \nu_j $ and $ \deg_y (f_i) = n_{(u)}(f_i) =: \nu_i $.
	Assume that $ (g_1, \ldots, g_{i-1}) $ is strongly normalized.

	There are elements $ x_{i,j} \in S[y] \cap \langle u \rangle $ with $ \deg_y(x_{i,j}) \leq \nu_i- \nu_j $, 
	for	$ 2 \leq i < j  \leq m $, such that
	$ (g_1, \ldots, g_{i-1}, g_i) $ is strongly normalized with respect to $ (y) $
	if we define 
	$ g_i := f_i - \sum\limits_{j=1}^{i-1} x_{i,j} g_j $.
\end{Prop}

As an immediate consequence, we see that $ \deg_y(g_i) = n_{(u)}(g_i) $.
For the proof, we introduce the following measure.

Recall that the total ordering $ \preceq $ on $ \IZ^r $ is defined by
$ A  \preceq B $ if and only if 
$ |A| < |B| $ or if 
$ |A| = |B| $ and $ A >_{lex} B $ (see \eqref{eq:def_prec_succ}). 

\begin{Def} 
	\label{Def:D(f)_C(f)}
	Let 
	\[  
	I := \langle \ini_0(g_1), \ldots, \ini_0(g_{i-1}) \rangle \subset k[Y_1, \ldots, Y_r].
	\]	
	Let $ f  = \sum\limits_{|B| \leq \nu } \vp_B \, y^B \in S[y] $ with $ \vp_B \in S $ and $ \nu := \deg_y (f) $.
	Suppose that $ f \notin \langle u \rangle $.
	We define
	\[ 
	D(f) := \inf_\preceq \{ B \in \IZ^r_{\geq 0 } \mid 
	 \vp_B \neq 0  \, \wedge \, B \in \exp(I) \}.
	\]
\end{Def} 

\noindent 
If $ D(f) $ does not exist, then $ (g_1, \ldots, g_{i-1}, f ) $ is strongly normalized,
and vice versa.

\begin{Lem}
	\label{KeyLemma} 
	For every $ D \in \exp(I) $ with $ \nu := |D|  $, there exists an element $ h_D \in \langle g_1, \ldots , g_{i-1} \rangle \cap S[y]  $
	such that $ \exp(h_D) = D $ and
	\[ 
	h_D = y^D + \sum_{\substack{|B|\leq \nu \\ B \in \exp(I) \\ B \succ D}}  \psi_{D,B} \, y^B
	+ \sum_{\substack{|B|\leq \nu \\ B \notin \exp(I)}} \psi_{D,B} \, y^B,
	\ \
	\mbox{ for } 
	\psi_{D,B} \in S,
	\]
	\[\mbox{and } \ 
	h_D = \sum\limits_{j=1}^{i-1} s_{D,j}\, g_j,
	\  
	\mbox{ for } 
	s_{D,j} \in S[y]
	\mbox{ with } \deg_y(s_{D,j}) \leq \nu- \nu_j. 
	\]

\end{Lem}

\begin{proof}
	We argue by an increasing induction on $ D $.
	We begin with the case 
	$ D = \exp(g_1) = \min_\preceq \{ \exp(G) \mid G \in I \} $.
	Since $ \deg_y(g_1) = n_{(u)} (g_1) = \nu $, we have
	\[
	g_1 = y^{D} 
	+ 
	\sum\limits_{\substack{|B|\leq \nu \\ B \prec D}} 
	\vp_{1,B}  \, y^B 
	+ 
	\sum\limits_{\substack{|B|\leq \nu \\ B \succ D}}
	\vp_{1,B} \, y^B,
	\ \ \mbox{ for }
	\vp_{1,B} \in S
	\]
	and $ \vp_{1,B} \in \langle u \rangle \subset S $ if $ B \prec D $.
	(Without loss of generality, the coefficient of $ y^{D} $ is $ 1 $).
	For $ B \prec D $, we have $ B \notin \exp(I) $
	and therefore, we may take $ h_D := g_1 $.
	
	Suppose $ D \succ \exp(g_1) $ and assume that the statement of the lemma is true for all $ B \prec D $ with $ B \in \exp(I) $. 
	Recall that we have
	$ \exp(I) =
	\exp ( \langle \ini_0 (g_1), \ldots, \ini_0(g_{i-1}) \rangle ) 
	$,
	where the appearing ideal lies in $ k[Y_1, \ldots, Y_r] $.
	Since $ D \in \exp(I) $,
	there are $ P_{D,j} \in k[Y] $, which are homogeneous of degree $ |D| - \nu_j $ 
	(where $ \nu_j = n_{(u)}(g_j) $ and $ 1 \leq j < i $),
	such that
	$ \exp(H) = D $,
	for $ H := \sum\limits_{j=1}^{i-1} P_{D,j} \cdot \ini_0 (g_j) $.
	Note that $ \deg_Y(H) = |D|  $.
	
	Write $ P_{D,j} = \sum_{C} \lambda_{C,D,j} Y^C $, for $ \lambda_{C,D,j} \in k $.
	By construction, $ \lambda_{C,D,j} \neq 0 $ implies $ |C| = |D| - \nu_j $.
	For every $ \lambda_{C,D,j} \in k $, we choose $ \mu_{C,D,j} \in S $ such that 
	$ \mu_{C,D,j} \equiv \lambda_{C,D,j} \mod \langle u \rangle $.
	We define $ \rho_{D,j} := \sum_{C} \mu_{C,D,j} y^C $.
	Clearly, $ \rho_{D,j} \in S[y] $ and 
	\begin{equation} 
	\label{eq:pf_deg} 
		\deg_y (\rho_{D,j}) = |D|-\nu_j = \nu - \nu_j 
	\end{equation} 
	Let $ h := \sum\limits_{j=1}^{i-1} \rho_{D,j} g_j  \in S[y] $.
	We have $ \deg_y (h ) = |D| = \nu $ and  $ \exp(h) = D $. 
	By definition of the leading exponent (see~\eqref{eq:def_exp}), we have 
	\[
	h=
	y^D 
	+ \sum_{\substack{|B|\leq \nu \\ B \in \exp(I) \\ B \prec D}}  \psi_{D,B} \, y^B
	+ \sum_{\substack{|B|\leq \nu \\ B \in \exp(I) \\ B \succ D}}  \psi_{D,B} \, y^B
	+ \sum_{\substack{|B|\leq \nu \\ B \notin \exp(I)}} \psi_{D,B} \, y^B
	\]
	where
	$ \psi_{D,B} \in S $ and 
	\begin{equation} 
	\label{eq:coeff_in_u}
	\psi_{D,B} \in \langle u \rangle, 
	\ \ \mbox{ if } B \prec D.
	\end{equation}
	(Without loss of generality, the coefficient of $ y^{D} $ is $ 1 $).
	If the first sum is $ 0 $, we are done and define $ h_D := h $
	and $ s_{D,j} := \rho_{D,j} $.
	
	Assume this is not the case,
	i.e., $ B_0 := D(h) \prec D $ (Definition~\ref{Def:D(f)_C(f)}).
	By the induction hypothesis,
	there exists $ h_{B_0} $ such that
	\begin{itemize}
		\item 
		$ h_{B_0} = y^{B_0} + \sum\limits_{\substack{|B|\leq \nu \\ B \in \exp(I) \\ B \succ {B_0}}}  \psi_{{B_0},B} \, y^B
		+ \sum\limits_{\substack{|B|\leq \nu \\ B \notin \exp(I)}} \psi_{{B_0},B} \, y^B $,
		for 
		$ \psi_{{B_0},B} \in S $,
		and 
		
		\item 
		$ h_{B_0} = \sum\limits_{j=1}^{i-1} s_{B_0,j} g_j $,
		for
		$ s_{B_0,j} \in S[y] $
		with $ \deg_y(s_{B_0,j}) \leq \nu- \nu_j $.
	\end{itemize}
	We eliminate the $ y^{B_0} $-term in $ h $ using $ h_{B_0} $:
	\[
	\begin{array}{rl}
		h^{(1)}  
		& :=  
		h - \psi_{D,B_0} h_{B_0} =
		\\[5pt]
		& = 
		\Big( 
		y^D
		+  \psi_{D,B_0} \, y^{B_0}
		+ \sum\limits_{\substack{|B|\leq \nu \\ B \in \exp(I) \\ 
				B_0 \prec B \prec D}}  
			\psi_{D,B} \, y^B
	+ \sum\limits_{\substack{|B|\leq \nu \\ B \in \exp(I) \mbox{ \tiny and } B \succ D
	\\
\mbox{\tiny or } B \notin \exp(I) }}  
	\hspace{-10pt} 
	\psi_{D,B} \, y^B	
	\Big) 
	- 
	\\[30pt]
	&
	\ \ \ 
	-
	\psi_{D,B_0} \Big(
		y^{B_0} + \sum\limits_{\substack{|B|\leq \nu \\ B \in \exp(I) \\ B \succ {B_0}}}  \psi_{{B_0},B} \, y^B
		+ \sum\limits_{\substack{|B|\leq \nu \\ B \notin \exp(I)}} \psi_{{B_0},B} \, y^B
		\Big)
	\\[40pt]
	& = 
	(1 - \psi_{D,B_0} \psi_{B_0,D}) 
	y^D
	+ 
	\sum\limits_{\substack{|B|\leq \nu \\ B \in \exp(I) \\ 
			B_0 \prec B \prec D}}  
	(\psi_{D,B}  - \psi_{D,B_0} \psi_{B_0,B})
	y^B
	\\[30pt]
	& \ \ \ 
	+ \sum\limits_{\substack{|B|\leq \nu \\ B \in \exp(I) \mbox{ \tiny and } B \succ D
			\\
			\mbox{\tiny or } B \notin \exp(I) }}  
	\hspace{-10pt} 
	(\psi_{D,B}  - \psi_{D,B_0} \psi_{B_0,B})
	y^B	
	\end{array} 
	\]
	Since $ 1 - \psi_{D,B_0} \psi_{B_0,D}, \psi_{D,B}  - \psi_{D,B_0} \psi_{B_0,B} \in S $, 
	we obtain 
	\[  
		D(h^{(1)}) \succ D(h) .
	\]  
	By \eqref{eq:coeff_in_u}, we have
	$ \psi_{D,B_0} \in \langle u \rangle $
	and hence $ 1 - \psi_{D,B_0} \psi_{B_0,D} \in S^\times $ is a unit. 
	
	Furthermore, 
	using $ h = \sum\limits_{j=1}^{i-1} \rho_{D,j} g_j $ 
	and $ h_{B_0} = \sum\limits_{j=1}^{i-1} s_{B_0,j}\, g_j $,
	we have
	\[
	h^{(1)}  = 
	\sum\limits_{j=1}^{i-1} \rho_{D,j} g_j 
	- \psi_{D,B_0} (\sum\limits_{j=1}^{i-1} s_{B_0,j}  g_j)
	=  
	\sum\limits_{j=1}^{i-1} ( \underbrace{\rho_{D,j} - \psi_{D,B_0} s_{B_0,j} }_{\displaystyle =: \rho_{D,j}^{(1)}} ) g_j ,
	\]
	where \eqref{eq:pf_deg}, $ \psi_{D,B_0} \in S $, and $ \deg_y(s_{B_0,j}) \leq \nu- \nu_j $
	imply $ \deg_y ( \rho_{D,j}^{(1)} ) \leq \nu - \nu_j $.

	Repeating the elimination step finitely many times, we eventually reach $ h^{(N)} $ 
	with $ D(h^{(N)}) = D $, for some $ N \in \IZ_+ $. 
	Using \eqref{eq:coeff_in_u}, we see that the coefficient of $ y^D $ in an expansion  of $ h^{(N)} $ 
	is of the form $ \psi = 1 + \psi_+ $, 
	for $ \psi_+ \in \langle u \rangle \subset S $. 
	Moreover, we have 
	$ h^{(N)} = \sum\limits_{j=1}^{i-1} \rho_{D,j}^{(N)} \, g_j $,
	for
	$ \rho_{D,j}^{(N)} \in S[y] $
	with $ \deg_y(\rho_{D,j}^{(N)}) \leq \nu- \nu_j $.
	Thus, $ h_D := \psi^{-1} h^{(N)} $ fulfills all desired properties. 
\end{proof}

\medskip

\begin{proof}[Proof of Proposition \ref{Prop:m-1_to_m}]
	If $ ( g_1, \ldots, g_{i-1}, f_i) $ is strongly normalized, then the statement holds for $ g_i := f_i $.
	Suppose this is not the case. 
	Let $ f := f_i $, $ \nu := n_{(u)} (f) = \deg_y (f) $, and $ D := D(f) $. 
	We have
	\[
	f = \vp_D \, y^D 
	+ 
	\sum\limits_{\substack{|B|\leq \nu \\ B \in \exp(I) \\ B \succ D}} 
	\vp_B  \, y^B 
	+ 
	\sum\limits_{\substack{|B|\leq \nu \\ B \notin \exp(I)}}
	\vp_B \, y^B,
	\]
	where 
	$ \vp_D, \vp_B \in S $ and  $ \vp_D\in \langle u \rangle$, 
	$ \vp_D \neq 0  $.
	
	Since $ D = D(f) \in \exp(I) $, Lemma \ref{KeyLemma} implies that there exists an element
	$ h_D \in \langle g_1, \ldots, g_{i-1} \rangle \cap S[y] $ 
	such that $ \exp(h_D) = D $,
	\begin{itemize}
		\item[(i)] 
	$ h_D = y^D + \sum\limits_{\substack{|B|\leq \nu \\ B \in \exp(I) \\ B \succ D}}  \psi_{D,B} \, y^B
	+ \sum\limits_{\substack{|B|\leq \nu \\ B \notin \exp(I)}} \psi_{D,B} \, y^B,
	$
	for $ \psi_{D,B} \in S $, and

		\item[(ii)] 
		$ h_D = \sum\limits_{j=1}^{i-1} s_{D,j} g_j $,
		for
		$ s_{D,j} \in S[y] $ with 
		$ \deg_y(s_{D,j}) \leq \nu- \nu_j $.
	\end{itemize}
	We define 
	\[
	f^{(1)} := f - \vp_D \, h_D = 
	\]
	\[ 
	= 
	\sum\limits_{\substack{|B|\leq \nu \\ B \in \exp(I) \\ B \succ D}} 
	(\vp_B - \vp_D \psi_{D,B} ) \,  y^B 
	+ 
	\sum\limits_{\substack{|B|\leq \nu \\ B \notin \exp(I)}}
	(\vp_B - \vp_D \psi_{D,B}) \, y^B
	.
	\]
	Since $ (g_1, \ldots, g_{i-1}, f_i ) $ is $ 0 $-normalized, 
	 we have $ \deg_y(f^{(1)}) = \deg_y (f) = n_{(u)} (f) $. 
	
	If $ (g_1, \ldots, g_{i-1}, f^{(1)} ) $ is strongly normalized, we define $ g_i := f^{(1)} $ and stop.
	Else, $ D(f^{(1)}) $ is defined and $ D(f^{(1)})  \succ D $
	as $ \vp_B - \vp_D \,\psi_{D,B} \in S $. 
	We repeat the elimination step finitely many times to obtain $ f^{(N)} $
	such that $ (g_1, \ldots, g_{i-1}, f^{(N)} ) $ is strongly normalized.
	Set $ g_i := f^{(N)} $.

	By (ii) and since $ \varphi_D \in S $, we have that
	$ g_i = f_i - \sum\limits_{j=1}^{i-1} x_{i,j} g_j $,
	for certain $ x_{i,j} \in S[y] $ with $ \deg_y (x_{i,j}) \leq \nu_i - \nu_j $.
	As $ (g_1, \ldots, g_{i-1}, f_i ) $ is $ 0 $-normalized, 
	we have $ \vp_D\in \langle u \rangle$,
	and hence, 
	we get $ x_{i,j} \in \langle u \rangle$.
\end{proof}

\smallskip

Let us give an example illustrating why Lemma~\ref{KeyLemma} is important in order to obtain a strict decrease for $ D( \ ) $ in the proof of Proposition~\ref{Prop:m-1_to_m} (and hence of Theorem~\ref{Thm:Strongly_Nlz}).

\begin{Ex}
	\label{Ex:why_Key_Lemma}
	Let $ k $ be a field of characteristic two, $ \car(k) = 2 $
	and let $ a, b, c, d, e \in \IZ_+ $.
	Consider the regular local ring $ R := k [u, y_1, y_2, z]_{\langle u, y_1, y_2, z \rangle } $ 
	and let $ J := \langle g_1, g_2, f \rangle \subset R $ be the ideal generated by
	\[ 
	g_1 = y_2^3 + u^a y_1^3  + u^b,
	\ \ \ \
	g_2 = y_1^4 + u^c y_1^2 y_2^2 + u^d, 
	\ \ \ \
	f = z^5 + u^e y_1 y_2^4.
	\]
	We observe that $ (g_1, g_2,f) $ is $ 0 $-normalized, but not strongly normalized.
	We have $ \nu = n_{(u)}(f) = 5 $ and $ D(f) = (1,4,0) $.
	If we use 
	$ y_1 y_2 g_1  = y_1 y_2^4 + u^a y_1^4 y_2  + u^b y_1 y_2 $ to eliminate $ u^e y_1 y_2^4 $ in  $ f $, we obtain
	\[
	f'^{(1)} := f - u^e y_1 y_2 g_1 = 
	z^5 +  u^{a+e} y_1^4 y_2  + u^{b+e} y_1 y_2.
	\] 
	{We have $ D(f'^{(1)}) = (4,1,0) \prec (1,4,0) = D(f) $. 
	Hence,}
	the induction on the value of $ D(.) $ breaks.
	
	The reason for this is that {$ h : =  y_1 y_2 g_1 $} 
	is not the appropriate element to eliminate 
	{$ u^{e} y_1 y_2^4 $} 
	in {$ f $} as there appears {$ u^a y_1^4 y_2 $}
	and $ (4,1,0) \prec (1,4,0) $.
	We first have to modify $ h $ using {$ g_2 $} (and ${ g_1 } $) in order to eliminate {$ u^a y_1^4 y_2 $} (and the bad terms possibly appearing afterwards) as in the proof of Lemma~\ref{KeyLemma}. 
	We leave the details and the end of the normalization process as an exercise to the reader.
\end{Ex}

Even though, the controlling vector $ D ( \ ) $ is not behaving well, 
one can check that applying Hironaka's normalization procedure eventually leads to a strongly normalized system in the previous example.

The following example show that being polynomial in $ (y) $ is not enough for the strong normalization to exist.
More precisely, we cannot skip the assumption $ \deg_y(f_i) = n_{(u)}(f_i) $. 

\begin{Ex}
	\label{Ex:normal_infinite}
	Let $ R $ be a regular local ring with regular system of parameters $ (u,y_1,y_2) $,
	(e.g., $ R = k[u,y_1,y_2]_{\langle u,y_1,y_2\rangle } $, for any field $ k $).
	Consider the ideal $ J := \langle f_1, f_2 \rangle \subset R $, where
	\[ 
	f_1 = y_1^3 - y_1^4 u - y_1^2 u^2 - u^5
			\hspace{10pt}\mbox{ and }\hspace{10pt}
			f_2 = y_2^5 + y_1^3 u.
	\]
	Clearly, $ (f_1, f_2) $ is $ 0 $-normalized. 
	The normalization process tells us to replace $ y_1^3 = f_1 + y_1^4 u + y_1^2 u^2 + u^5 $ in $ f_2 $ and we obtain
	\[
	g_2 = y_2^5 + y_1^4 u^2 + y_1^2 u^3 + u^6.
	\]
	The monomial $ y_1^2 u^3 $ yields the vertex $ 1 \in \poly{ f_1, g_2}{ u }{ y_1, y_2 } $ which does not vanish by further normalization and which is not solvable.
			Thus we have $ \poly{ f_1, g_2}{ u }{ y_1, y_2 } = \cpoly{ J }{ u } $.
			\\ \indent 
			But $ ( f_1, g_2 ) $ is not normalized in our stronger sense. 
			Again we would have to replace $  y_1^3 = f_1 + y_1^4 u + y_1^2 u^2 + u^5 $ and get
	\[
	h_2 = y_2^5 + y_1^5 u^3 + y_1^3 u^4 + y_1 u^7 + y_1^2 u^3 + u^6.
	\]
	Again, $ y_1^3 $ appears and we run into a loop.
	But the polyhedron does not change anymore. 
\end{Ex}

\begin{Rk}
	In order to obtain the characteristic polyhedron, Hironaka only wants a system of generators that is normalized at every vertex.
	Hence, the more restrictive notion of strongly normalization is not needed in the general case. 
	
	Let $ R $ be a regular local $ G $-ring.
	Let $ (f,u,y) $ be such that $ \poly fuy $ is defined and non-empty.
	If we pass to the data given by the initial forms at any compact face of $ \poly fuy $, then hypothesis~$ (\Pol ) $ holds.
	Hence, one can apply strong normalization there. 
\end{Rk}

\smallskip 

For the proof of Theorem~\ref{Thm:Strongly_Nlz}, it remains to show

\begin{Lem} 
	\label{Lem:remains(u)std}
	Let $ R $ be a regular local ring,
	$ ( u , y ) = ( u_1, \ldots, u_e; y_1, \ldots, y_r ) $ be a regular system of parameters of $ R $, $ J \subset R $ a non-zero ideal,
	and $ (f) = (f_1, \ldots, f_m) $ be a $ 0 $-normalized $ ( u ) $-standard basis of $ J $.
	Let $ ( g) = (g_1, \ldots, g_m) $ be the strongly normalized elements obtained by the preceding normalization process.
	Then, $ ( g_1, \ldots, g_m) $ is also a $ ( u ) $-standard basis for $ J $.
\end{Lem} 

\begin{proof}
	Let $ (y,L ) $ be a reference datum for the $ ( u ) $-standard basis (Definition~\ref{Def:u-standardbasis}).
	Since $ ( f ) $ is $ 0 $-normalized, our construction does not modify the $ 0 $-initial forms, $ \ini_0(f_i) = \ini_0(g_i) $, for $ 1 \leq i \leq m $.
	Hence, properties (3)(a) and (3)(b) in Definition~\ref{Def:u-standardbasis} also hold for $ (g) $.
	Further, the construction of the elements $ h_D $ in the proof of Lemma \ref{KeyLemma}
	and the property $ \ini_L(f_i) = \ini_0 (f_i) $, for $ 1 \leq i \leq m $,
	imply that we also have $ \ini_L(g_i) = \ini_0 (g_i) $.
	Hence, $ (g) $ is a $ (u) $-effective basis, which remained to be proven.
	The assertion follows. 
\end{proof}

We end the subsection with a useful stability result.

\begin{Prop}
	\label{Prop:(Pol)_strong_nlzd_stable}
	Let $ (S,N,k) $ be a regular local ring with regular system of parameters $ (u) = (u_1, \ldots, u_e ) $,
	let $ R = S[y_1, \ldots, y_r]_M $, for $ M = \langle u , y \rangle $
	and $ J \subset R $ be a non-zero ideal.
	Let $ ( f ) = (f)_{(u,y)} = (f_1, \ldots, f_m) $ be a strongly normalized $ (u) $-standard basis for $ J $
	with $ f_i \in S[y] $ and $ deg_y(f_i) = n_{(u)}(f_i) $,
	for $ 1 \leq i \leq m $.
	
	Consider elements $ ( \phi  ) = ( \phi_1, \ldots, \phi_r ) \in N^r \subset S^r $.
	If we define
	\[ 
	z_j := y_j + \phi_j, \ \ \ \mbox{ for } 1 \leq j \leq r ,
	\] 
	then $ (f)_{(u,z)} $ is still a strongly normalized $ (u) $-standard basis for $ J $.
\end{Prop}

\noindent 
In particular,
if $ (\Pol) $ holds, then the property of being strongly normalized is stable under translations of $ (y) $ by elements in $ N $.

\begin{proof}
	First, we observe, that the leading exponents do not change under such a translation. 
	Let $ u^A y^B $ be a monomial appearing in $ f_{i+1} $, for some $ 1 \leq i \leq m-1 $.
	It is mapped to
	\[
	\begin{array}{c} 
	u^A (z_1 + \phi_1)^{B_1} \cdots (z_r + \phi_r)^{B_r} 
	=
	\\[10pt]
	=
	\displaystyle 
	u^A \sum\limits_{C_1 = 0}^{B_1} \ldots \sum\limits_{C_r = 0}^{B_r} 
	\binom{B_1}{C_1} \cdots \binom{B_r}{C_r} 
	\phi_1^{B_1 - C_1} \cdots \phi_r^{B_r - C_r}
	z_1^{C_1} \cdots z_r^{C_r}.
	\end{array} 
	\]
	In particular, $ B \in (C_1, \ldots, C_r ) + \IZ^r_{\geq 0 } $, for every appearing $ C = (C_1, \ldots, C_r ) $.
	So, if $ C \in \exp(\langle f_1, \ldots, f_i \rangle ) $, then $ B \in \exp(\langle f_1, \ldots, f_i \rangle ) $.
	Hence,
	this would contradict the assumption that $ (f)_{(u,y)} $ is strongly normalized. 	
\end{proof}

\medskip

\subsection{Finding Coordinates under hypothesis $(\Pol) $}	
We come to task of finding suitable coordinates $ (z_1, \ldots, z_r) $ if the characteristic polyhedron is empty and $ (\Pol) $ holds,
i.e., such that $ \poly Juz = \cpoly Ju = \varnothing $.
Hypothesis $ (\Pol) $ allows to prove a slightly stronger result, namely, the elements $ (z_1, \ldots, z_r )$ can be obtained by a translation in $ N \subset S $.
Using Proposition~\ref{Prop:(Pol)_strong_nlzd_stable}, we then obtain $ \poly fuz =  \varnothing $,
for a strongly normalized standard basis $ (f) $.  

The following result generalizes \cite{CPcompl} Corollary~3.4.

\begin{Thm}[Claim~\ref{claim:polyvide} for $(\Pol)$]
	\label{Thm:Claim_2.3_pol}
	
	Let $ (S,N,k) $ be a regular local $G$-ring 
	with regular system of parameters $ (u) = (u_1, \ldots, u_e) $.
	Let	$ R = S[y_1,\ldots,y_r]_{M} $ and $ J \subset R $ be a non-zero ideal
	such that the directrix of $ J'$ be determined by $ ( y ) $
	and such that $ (u) $ is a regular $ (R/J) $-sequence.
	Let $ (f) = (f_1, \ldots, f_m ) $ be a $ 0 $-normalized $ (u) $-standard basis of $ J $ with $ f_i\in S[y_1,\ldots,y_r] $ 
	and
	$ \deg_{y}(f_i) = n_{(u)} (f_i) =: \nu_i $,
	for $ 1 \leq i \leq m $.

	Suppose $ \cpoly Ju = \varnothing $ and 
	let $ ( \hy ) = ( \hy_1 ,\ldots, \hy_r ) $ be the elements in $ \hR $ such that $ \poly Ju\hy = \varnothing $ 
	obtained by Hironaka's vertex preparation. 
	There exist 
	$ x_{i,j} \in I( \poly fuy ; \nu_{i}-\nu_j) \cap S[y] $,
	with $ \deg_y (x_{i,j}) \leq \nu_i - \nu_j $, 
	$ 2 \leq j < i \leq m $, 
	and
	$ ( \varphi ) = ( \varphi_1, \ldots, \varphi_r ) $ in $ N \subset S  $ 
	such that,
	if we set 
	\[ g_i = f_i + \sum\limits_{j=1}^{i-1}  x_{i,j} f_j \in  R 
	\ \ 
	\mbox{ and } 
	\ \
	z_a := y_a + \varphi_a, 
	\]
	for $ i \in \{ 1, \ldots, m \} $ and $ a \in \{ 1 , \ldots, r \} $,
	then the following properties hold:
	
		\begin{enumerate}
		\item 	
		$ ( u, z) $ is a regular system of parameters for $  R $,
		
		\item 
		$ ( z ) $ yields the directrix of $ J' = J \cdot R/\langle u \rangle  $,
		
		\item 
		$ \langle z \rangle \cdot \widehat{ R} = \langle \hy \rangle $, 
		
		\item $ z_a \in I(  \poly fuy;1)  $, for $ 1\leq a \leq r $,
		
		\item 
		$ (g_1, \ldots, g_m) $ is a strongly normalized $ ( u ) $-standard basis,
		and
		
		\item 
		$ 	\poly guz = \poly Juz = \cpoly Ju = \varnothing . $
	\end{enumerate} 
\end{Thm}

\begin{proof} 
	Without loss generality, we assume that $ (f_1, \ldots, f_m) $ is strongly normalized. Suppose we could find $ (z ) $ verifying (1)(2)(3)(4), as $ (z) $ is obtained as a translation by an element in $ S $,
		by Proposition~\ref{Prop:(Pol)_strong_nlzd_stable},  the system $ (f_1, \ldots, f_m ) $ fulfills the property (5) and (6).
	
	Therefore, it remains to find $ (z) $ fulfilling (1)(2)(3)(4).

	\smallskip

	Recall that $ k = S / N $ denotes the residue field of $ S $ and that we fixed a strongly normalized $ ( u ) $-standard basis 
		$ (f) = (f_1, \ldots, f_m) $ with $ f_i \in S[y_1, \ldots, y_r] $ and $ \nu_i = \deg_y (f_i) = n_{(u)} (f_i) $.
		Since the characteristic polyhedron is empty, applying finitely many times Hironaka's procedure, we may suppose $\vert v\vert >1$ for all $v \in \poly fuy$:
		the elements $ (f) $ are also a standard basis for $ J $  in the sense of Definition~\ref{Def:u-standardbasis} 
		and their $ 0 $-initial forms coincide with the initial forms at the maximal ideal, $ \ini_0 (f_i) = \ini_M(f_i) $, for $ 1 \leq i \leq m $.
		We have that
		\[
		f_i = F_i (y) + \sum\limits_{B \in \IZ_{\geq 0}^r : |B| < \nu_i } s_{B,i}\, y^B,
		\ \ \ \mbox{ for } 1 \leq i \leq m ,
		\]
		where $  s_{B,i} \in S  $ and $ F_i (y) = \sum\limits_{|B|=\nu_i} c_{B,i} y^B \in S[y] $ are polynomials homogeneous of degree $ \nu_i $,
		with $ c_{B,i} \in S $ (not all necessarily units).

	We use the notation
		\[
		\IK := \mathrm{Quot}(S)
		\ \ \ \mbox{ and } \ \ \
		\IL := \mathrm{Quot}(\hS).
		\]
	Consider the homogenizations $ ( f^h ) = (f_1^h, \ldots, f_m^h ) \in S[y,T]^m $ as elements in $ \IK[y_1, \ldots, y_r, T] $,
		\[
		f_i ^h= F_i (y) + \sum\limits_{B \in \IZ_{\geq 0}^r : |B| < \nu_i } s_{B,i}\, y^BT^{\nu_i- |B|},
		\ \ \ \mbox{ for } 1 \leq i \leq m .
		\]
	Recall that by Hironaka's procedure $ \hy_j = y_j + \phi_j $, for  $ 1 \leq j \leq r $ and where $ \phi_1,\ldots,\phi_r \in \hS $.
	As there is no normalization to do, by Proposition~\ref{Prop:(Pol)_strong_nlzd_stable}, we get that
		\[
		f_i ^h= F_i (y_1+\phi_1 T,\ldots,y_r+\phi_r T) 
		\ \ \ \mbox{ for } 1 \leq i \leq m .
		\]
		In other words, the ideal of the directrix of the cone over $ \IL $ defined by
		$\prod_{1\leq i \leq m}f_i^h\in \IL[
		y_1,\ldots,y_r,T]$ is contained in 
		$ \langle  y_1+\phi_1 T, \ldots, y_r+\phi_r T \rangle $.
		By \cite{CJS} Theorem~2.7, applied for%
		\footnote{Here, we use the notation of \cite{CJS} Theorem~2.7 for clarity in the reference.
		We warn the reader not to get confused with the notation of the present paper.}
		$ R = \hS[y_1,\dots,y_r]_{\langle u, y_1,\dots,y_r \rangle} $,
		the principal ideal 
		$ J = \langle f_1\cdots f_m \rangle \subset R $, 
		and
		$ \mathfrak{p} = \langle \hy_1, \ldots, \hy_r \rangle $, 
		we have 
		\[
			e(R_{\mathfrak p}/ J_{\mathfrak p})\leq e(R/J)-\dim(R/\mathfrak{p}),
		\]
		where $e( \ )$ means the dimension of the directrix at the maximal ideal (cf.~\cite{CJS} Definition~1.18).
		As $ \langle Y_1,\ldots,Y_r \rangle $ is the ideal of the directrix 
		at $ \langle u, y \rangle $ 
		of $ (f_1,\ldots,f_m) $ (considered as elements in $ \hS[y_1,\dots,y_r]_{\langle u, y_1,\dots,y_r \rangle}$),
		which is the same as the directrix of $ f_1 \cdots f_m $, 
		the right hand side of the above inequality is zero.
			Hence, we have 
			\[ 
				e(R_{\mathfrak p}/ J_{\mathfrak p})=0.
			\]
		 In other words, the ideal of the directrix of $ f_1 \cdots f_m $ at $ \mathfrak{p} $ must have $ r $ generators.
		 Therefore the ideal of the directrix of the cone defined by $\prod_{1\leq i \leq m}f_i^h\in \IL[
		y_1,\cdots,y_r,T] $ is $ \langle y_1+\phi_1 T,\cdots,y_r+\phi_r T \rangle $.
		Therefore, the elements $ \phi_1, \ldots, \phi_r $ are uniquely determined in $ \hS $.
		 
		\smallskip 
		 
		Since $ S $ is a $ G $-ring, the extension $ \IK \subset \IL $ is separable, so  we have $\phi_1,\ldots,\phi_r \in \hS \cap  \IK = \hS \cap  \mathrm{Quot}(S) = S $ (the last equality holds by faithful flatness). 
		As  $\phi_1,\ldots,\phi_r $ are the one given by Hironaka's procedure,  
		the system $ (z) = (z_1, \ldots, z_r ) = (y_1+\phi_1, \ldots, y_r + \phi_r ) $, which lies in $ S[y]_{\langle u, y \rangle }$, fulfills (1)(2)(3)(4).
\end{proof}

Let us emphasize that $ \phi_1, \ldots, \phi_r \in R $ are the elements that we obtain from Hironaka.
In theory, as the directrix may be computed (see the appendix of \cite{CJSc}),
there is a {finite} way to compute the elements $ \phi_1, \ldots, \phi_r $ given by Hironaka's procedure that may be infinite.

\smallskip

Note that Theorems~\ref{Thm:ReducEmpty} and~\ref{Thm:Claim_2.3_pol} imply Theorem~\ref{MainThm:Pol}.

%
%
%
%
%
%
%
%
%
%

\medskip 

\section{Empty Characteristic Polyhedron II: Hypothesis $ (*) $}
\label{subsec:empty(*)}
\label{sec:empty*}

Our next goal is to show that we find suitable elements $ (z) = (z_1, \ldots, z_r ) $ in $ R $ extending $ (u) $ to a regular system of parameters,
under the assumption that $ \cpoly Ju = \varnothing $ and hypothesis $ (*) $ is true.

\begin{Prop}
	\label{Prop:EmptyCaseSpecial}
	Let $ R $ be a regular local G-ring, $ J \subset R $ be a non-zero ideal and $ ( u , y ) = ( u_1, \ldots, u_e; y_1, \ldots, y_r ) $ be a regular system of parameters~of $ R $  such that 
	$ (u) $ is a regular $ (R/J) $-sequence
	and 
	$ ( y ) $ determines the directrix of $ J' = J \cdot R' $, where $ R' = R / \langle u \rangle $. 
	Suppose $ \cpoly Ju = \varnothing $ and 
	let $ ( \hy ) = ( \hy_1 ,\ldots, \hy_r ) $ be the elements in $ \hR $ such that $ \poly Ju\hy = \varnothing $ 
	obtained by Hironaka's vertex preparation. 

	We assume that condition $ (*) $ holds for $ (J,R,u) $, i.e., one of the following conditions is true:
	\begin{enumerate}
		\item[(a)]  the dimension of the ridge of $ C_u(J) $ coincides with the dimension of its directrix,
		\item[(b)] or $ \car ( k ) \geq \dfrac{\dim (X)}{ 2} + 1 $, 
		where $ X := \Spec(R/J) $.
	\end{enumerate}
	
	There exist elements $ ( z ) = ( z_1, \ldots, z_r ) $ in $ R $ such that $ ( u, z) $ is a regular system of parameters~for $ R $, $ ( z ) $ yields the directrix of $ J' $, 
	$ \langle z \rangle \cdot \hR = \langle \hy \rangle \subset \hR$,
	and
	$
	\poly Juz = \cpoly Ju = \varnothing .
	$
\end{Prop}

Recall that 
a subscheme $ D \subset X := \Spec (R/J) $ is called a {\em permissible center} for $ X $ if it is regular and $ X $ is normally flat along $ D $ at every point of $ D $ (\cite{CJS} Definition 2.1).
The latter holds if $ D $ is contained in the Hilbert-Samuel locus of $ X $, $ D \subset \HS( X) $, 
see loc.~cit., Theorem~2.3.

\begin{proof}[Proof of Proposition \ref{Prop:EmptyCaseSpecial}]
	Due to Hironaka (Theorem \ref{Thm:Hironaka}) there are $ ( \hf, \hy ) $ in $ \hR $ such that we have $ \poly Ju\hy = \poly{\hf}u\hy = \varnothing $.
	Since $ \poly Ju\hy = \varnothing $ we get that $ \widehat{D} := \Spec ( \hR/ \langle \hy \rangle  ) $ is a permissible center for $ \widehat{X} := \Spec(\hR/\hJ) $, $ \hJ = J \cdot \hR $. 
	This follows from \cite{CJS} Theorem~2.2(2) using $ ( \hf ) $ 
	and the definition of $ \poly{\hf}u\hy $ (Definition~\ref{Def:asso_poly})
	in part (iv) of the cited theorem.
	
	\begin{Claim}\label{claim:bupolyvide}
	Assumption $(*)$ implies that, after blowing up along $ \widehat{D}$, the Hilbert-Samuel function decreases at every point lying above the center.
	\end{Claim}
	
	Take it for granted for the moment.
	
	By Claim~\ref{claim:bupolyvide}, the maximal value of the Hilbert-Samuel function strictly decreases after the blowing up with center $ \widehat D $.
		Since a blowing up is an isomorphism outside of the center, this may only happen if
	we have blown up the entire maximal Hilbert-Samuel stratum $ \HS(\hX) $ of $ \widehat{X} $: $I_{\HS(\hX)} = \langle \, \hy \, \rangle \hR $, 
	where $ I_{\HS(\hX)} $
		denotes the ideal of the Hilbert-Samuel stratum of $ \hX $.

	By \cite{CJS} Lemma 1.37(2), 
	the Hilbert-Samuel stratum of $ \hX $
	(which is given by $ \hJ= J \hR $) 
	is solely determined by that of $ X=  \Spec(R/J) $,
	\begin{equation}\label{eq:HSstratum}
	I_{\HS(\hX)} = I_{\HS(X)} \cdot \hR \ = \langle \, \hy \, \rangle \hR,
	\end{equation}
	where $ I_{\HS(X)}  $ 
	denotes the ideal of the Hilbert-Samuel stratum of $ X $.
	Hence there is an ideal $ I \subset R $ such that $ I \cdot\hR = \langle \, \hy \, \rangle $.
	Since $ R $ is excellent (Lemma \ref{lem:G=>excellent}), $ R / I  $ is regular and the height of $ I $ is $ r $.
	Thus there exist regular elements $ ( z ) = ( z_1 , \ldots,  z_r ) $ in $ R $ such that $ I = \langle z_1, \ldots, z_r \rangle $.
	This implies 
	\begin{equation}
	\label{eq:zparameters} 
	\langle z \rangle \cdot \hR = I\cdot \hR = \langle \hy \rangle .
	\end{equation}

	Therefore $ ( u, z ) $ is a regular system of parameters~for $ R $, $ ( z ) $ determines the directrix of $ J' $
	and we have the desired equality $ \poly{ J }{ u }{ z } = \poly{ J }{ u }{ \hy } = \cpoly{ J }{ u }  = \varnothing $.
	
	\smallskip 
	
	\noindent 
	{\em Proof of the Claim~\ref{claim:bupolyvide}.}
	Let $ \Dir_{\hx}(\hX) $ (resp.~$ \Rid_{\hx}(\hX) $) be the directrix (resp.~the ridge) of the cone defined by $ \ini_{\widehat{M}}(J\hR ) $, $\hx$ is the closed point in $\Spec(\hR)$.
	
	If  $(*)(a)$ holds, then 
	\begin{equation}
	\label{eq:ridgereduced} 
	 (\Rid_{\hx}(\hX))_{red} = \Dir_{\hx} (\hX) 
	\end{equation}
	which follows using \cite{GiraudEtude} Proposition~5.4(i), p.I-27.
		
	It follows from \cite{GiraudEtude} Corollaire~2.4, p.III-13 
	(see also \cite{BernhardThesis} Theorem~(9.2.2)) 
	that if $ \hx ' $ is near to $ \hx $ then 
		$ \hx' $ must be contained in the projective space associated to the quotient group $ \Rid_{\hx}(\hX)/T_{\hx}(D) $,
		where $ T_{\hx} ( D) $ denotes the Zariski tangent space of $ D $ at $ {\hx}$.  By \eqref{eq:ridgereduced}, this  projective space is empty.

		In case $(*)(b)$, this is a consequence of a result by Hironaka \cite{HiroAdd} Theorem~2 which was later improved by Mizutani \cite{Miz} (see also \cite{BernhardThesis} section 10.9).
		This is in connection with so called Hironaka group schemes (e.g., see \cite{GiraudEtude} {\S~2, {p.III-11}}, 
		or \cite{BernhardThesis} chapter 9).
		For a detailed statement and proof, we refer to \cite{CJS} Theorem~2.14.
\end{proof} 

Note that this proves statements (1)(2)(3) of Claim~\ref{claim:polyvide} assuming hypothesis $ (*) $ to be true,
	while (4)(5)(6) will follow from Proposition \ref{Prop:EmptyCaseSpecial_new} below.

\begin{Obs}
	In fact, in order to apply the argument of the previous proof for finding the elements $ (z) = (z_1, \ldots, z_r) $, we only need that
	the ideal of the Hilbert Samuel stratum $ \HS(\hX) $ is $ \langle \hy_1, \ldots, \hy_r \rangle $ (and that $ R $ is a regular local $ G $-ring). 
	
	This assumption is less restrictive than $ (*) $.
	For example, let $ R $ be a regular local $ G $-ring containing a non-perfect field $ k $ of characteristic two,
	and
	let
	$ J = \langle f \rangle $ be a principal ideal such that $ \dim ( \Spec(R/J)) \geq 3 $.
	If $ f = \hy_1^2 + \lambda \hy_2^2 $ in $ \hR $, 
	for $ \lambda \in k \setminus k^2 $, 
	then neither $ (*)(a) $ nor $ (*)(b) $ holds,
	while $ \HS(\hX) = \Spec (\hR/\langle \hy_1, \hy_2 \rangle ) $. 
\end{Obs}

\medskip

\section{Empty Characteristic Polyhedron III: Henselian Rings (following O.~Piltant)}
\label{sec:empty_hensel}

In this section, we discuss Theorem~\ref{MainThm:Hensel}, 
the case when $ R $ is a {\em Henselian} regular local $ G $-ring. 
The arguments for the proof have been outlined to us by Olivier Piltant in a private conversation. 

Analogous to the previous section, we prove

\begin{Thm}
	\label{Thm:coord_pol_hensel}
	Let $ R $ be a regular local G-ring, $ J \subset R $ be a non-zero ideal and $ ( u , y ) = ( u_1, \ldots, u_e; y_1, \ldots, y_r ) $ be a regular system of parameters of $ R $  such that 
	$ (u) $ is a regular $ (R/J) $-sequence
	and 
	$ ( y ) $ determines the directrix of $ J' = J \cdot R' $, where $ R' = R / \langle u \rangle $.
	 
	Suppose $ \cpoly Ju = \varnothing $ and $ R $ is Henselian. 
	Let $ ( \hy ) = ( \hy_1 ,\ldots, \hy_r ) $ be the elements in $ \hR $ such that $ \poly Ju\hy = \varnothing $ 
	obtained by Hironaka's vertex preparation. 

	There exist elements $ ( z ) = ( z_1, \ldots, z_r ) $ in $ R $ such that $ ( u, z) $ is a regular system of parameters~for $ R $, $ ( z ) $ yields the directrix of $ J' $, 
	$ \langle z \rangle \cdot \hR = \langle \hy \rangle \subset \hR$,
	and
	$
	\poly Juz = \cpoly Ju = \varnothing .
	$
\end{Thm}

The theorem provides statements (1)(2)(3) of Claim~\ref{claim:polyvide} under the assumption that $ R $ is Henselian.
The remaining parts (4)(5)(6) will follow from Proposition~\ref{Prop:EmptyCaseSpecial_new} below.

\begin{proof}
	We already know how to prove the result if the reduced ridge coincides with the directrix (Proposition~\ref{Prop:EmptyCaseSpecial}).
	Therefore we can exclude the case that the residue field $ k = R / M $ is perfect (in particular, we may assume $ \car ( R / M ) > 0 $) in the following.
	
	Due to Hironaka (Theorem~\ref{Thm:Hironaka}) there are $ ( \hy ) = ( \hy_1, \ldots, \hy_r ) $ in $ \hR $ 
	and a $ (u) $-standard basis $ (\hf ) = (\hf_1, \ldots, \hf_m) $ of $ \hJ = J \cdot \hR $
	such that 
	\[ 
	\poly \hf u\hy = \varnothing .
	\]
	The latter implies that 
	$ \hD := \Spec ( R / \langle \hy_1, \ldots, \hy_r \rangle ) $ 
	is a permissible center for $ \hX := \Spec(\hR/\hJ ) $
	(i.e., $ \hD $ is regular and $ \hX $ is normally flat along $ \hD $ at every point of $ \hD $, 
	\cite{CJS} Definition 2.1).
	In particular, $ \hD $ is contained in the Hilbert-Samuel locus of $ \hX $, 
	$ \hD \subset \HS(\hX ) $.
	
	Moreover, $ \poly \hf u\hy = \varnothing $ provides that, for all $ 1 \leq i \leq m $, we have
	\[ 
	\hf_i \in \langle \, \hy \, \rangle^{ \nu_i }  \, ,
	\]
	where $ \nu_i := \ord_M ( f_i ) = n_{(u)} (\hf_i) $,
	which follows using the definition of the polyhedron associated to $ (\hf,u,\hy) $ (Definition~\ref{Def:asso_poly}).

	We introduce the following ideals: 
	\[ 
	\begin{array}{c}
	\hfP := \langle \hy_1, \ldots, \hy_r \rangle \subset \hR
	\ \ \ \mbox{ and} \ \ \ 
	\fP : = \hfP \cap R = \langle \varphi_1 , \ldots, \varphi_d \rangle \subset R.
	\end{array} 
	\]
	By \cite{CJS} Lemma~1.37(2) the Hilbert-Samuel stratum of $ \widehat{J} = J \cdot \widehat{R} $ is solely determined by that of $ J $.
	Therefore, we have $ \fP \neq \langle 0 \rangle $.
	Furthermore, 
	we have a {\em unique} decomposition 
	\begin{equation}
	\label{decompP}
	\fP \hR = \fQ_1 \cap \ldots \cap \fQ_s \subset \hfP.
	\end{equation}
	where the $\fQ_i $ are prime ideals, $ 1 \leq i \leq s $ 
	(\cite{EGA_IV}~ (7.8.3.1)(vii): as $ {R/\fP} $ is reduced,  the completion  $ \hR / \fP \hR $ is reduced).
	The uniqueness implies that there are no embedded components.
	
	Since $ R $ is excellent (Lemma~\ref{lem:G=>excellent}),
	the fiber $ \Spec ( \hR_\hfP / \fP \hR_\hfP ) $ is regular and thus there exists a regular system of parameters for $ \hR_\hfP / \fP \hR_\hfP  $.
	Moreover,   a regular system of parameters for  $ \hfP \cdot \hR_\hfP \subset \hR_\hfP $ can be picked as the set consisting of the pull back of a regular system of parameters for $ \hR_\hfP / \fP \hR_\hfP $ and   a regular system of parameters for  $ \fP \cdot R_\fP = \langle \vp_1, \ldots, \vp_c \rangle_{R_{\fP}} \subset R_\fP $, extracted from a system of generators $ \langle \vp_1, \ldots, \vp_c, \ldots, \vp_d \rangle$ of $\fP \subset R$.
	Therefore, we may assume, without loss of generality,
	%
	\[ 
	\hfP \hR_\hfP  = \langle \, \vp_1, \ldots, \vp_c, \, \hy_{c + 1 }, \ldots, \hy_r \, \rangle_{ \hR_\hfP } \subset \hR_\hfP .
	\]
	and
	\[ 
	\gr_\fP ( R_\fP ) = \kappa ( \fP ) [ \phi_1, \ldots, \phi_c ],	
	\]
	where $ \phi_i := \ini_\fP ( \vp_i ) $, for $ 1 \leq i \leq c $.
	We consider the polynomial ring
	\[ 
	\gr_\hfP ( \hR_\hfP ) 
	\cong \kappa (\hfP) [ \tilde\phi_1, \ldots, \tilde\phi_c, \hY_{ c + 1 }, \ldots, \hY_r ]
	,
	\]
	where $ \hY_j := in_\hfP ( \hy_j ) $, for $ 1 \leq j \leq r $, 
	and $ \tilde\phi_i :=  \ini_\hfP ( \vp_i ) $, for $ 1 \leq i \leq c $.
	The natural map
	\[ 
	\gr_\fP ( R_\fP ) \to  \gr_\hfP ( \hR_\hfP ) 
	\]
	sends $\phi_i$ to $ \tilde\phi_i $, for $ 1 \leq i \leq c $, so it is injective
	and homogeneous of degree $1$. 
	In particular, this implies that %
	\[ 
		\deg_\hfP ( \tilde\phi_i ) =  \ord_\hfP ( \vp_i ) = 1,
		\ \ \
		\mbox{ for all }
		i \in \{ 1, \ldots, c \}.
	\]
	But so far we only know:
	\[ 
	\ord_M ( \vp_i ) \geq 1.
	\]
	
	
	Let us now come back to decomposition \eqref{decompP},
	$
	\fP \hR = \fQ_1 \cap \ldots \cap \fQ_s \subset \hfP.
	$
	
	\begin{Claim}
		\label{Claim:proof_hensel}
	The following properties hold:
		\begin{enumerate}
			\item 
			We have
			$ \fQ_i \cap R = \fP $,
			for all  $ i \in \{  1, \ldots, s \} $.
			
			\item 
			There is at least one $ i_0 $, without loss of generality $ i_0 = 1 $, such that 
			%
			%
			$ \fQ_1 \subset \hfP $.
			
			\item 
			We have 
			$ ( \, \fP \hR \, )_{\hfP} = \hfP \hR_{\hfP} $.
		\end{enumerate}
	\end{Claim}

	\begin{proof} 
		(1)
	Suppose $ \fP \varsubsetneq \fQ_i \cap R $.
	The flatness of $ R \to \hR $ implies that we have the going down property, 
	i.e.,
	there is an ideal $ \fQ' \subset \hR $ such that $ \fQ' \subsetneq \fQ_i $ and $ \fQ' \cap R = \fP $.
	But this contradicts the uniqueness of the decomposition (\ref{decompP}).

	\smallskip 
	
	(2)
	If we have $ \fQ_i \subset \hfP $, for some $ i \in \{ 2, \ldots, s \} $, we obtain the assertion after renaming the components.
	Hence, suppose that $ \fQ_i \not\subset \hfP $, for $ 2 \leq i \leq s $.
	Then, there exist elements $ a_i \in \fQ_i $ such that $ a_i \notin \hfP $.
	For every element $ a \in \fQ_1 $, we have
	$
	a \cdot a_2 \cdots a_s \in \hfP
	$
	and, 
	since $ \hfP $ is a prime ideal, we must have $ a \in \hfP $, 
	which provides $ \fQ_1 \subset \hfP $ as desired.
	
	\smallskip 
	
	(3) As $ J \subset R $, its directrix lies in 
	$ \gr_\fP ( R_\fP ) = \kappa (\fP) [ \phi_1, \ldots, \phi_c ] $.
		By the minimality of the directrix, 
		we have 
		\[ 
			e({\hR_{\hfP}/ J\hR_{\hfP}})
			:=
			\dim(\Dir({\hR_{\hfP} / J\hR_{\hfP}}))\geq r-c .
		\]
		On the other hand, \cite{CJS}~Theorem~2.7 provides  
		(Recall that $ \hfP = \langle \hy_1, \ldots, \hy_r \rangle $)
		\[ 
			r-c\leq e({\hR_{\hfP} / J\hR_{\hfP}}) 
			\leq 
			e({\hR / J\hR}) - \dim({\hR_{\hfP} / \hfP \hR_{\hfP} })
			=0.
		\]
		So, we have $ c = r $
		and $  \Spec ( \hR_\hfP / \fP \hR_\hfP ) $ is regular of dimension $0$,
		i.e., $ \hR_\hfP / \fP \hR_\hfP $ is a field.
	\end{proof}

	Thus, since $ ( \, \fP \hR \, )_{\hfP} = \hfP \hR_{\hfP} $ and  $ \fQ_1 \subset \hfP $ (Claim~\ref{Claim:proof_hensel}(2) and (3)), 
	we get 
	\[  
	\dim (\hR_\hfP / \hfP\hR_\hfP)  
	= \dim (\hR_\hfP / \fQ_1 \hR_\hfP)  
	= \dim (R_\fP / \fP R_\fP ).
	\]
	Since $ R $ is catenary, this implies
	$ 
	\dim (\hR / \hfP)  = \dim (\hR / \fQ_1) =  \dim (R / \fP)
	$.
	Hence, since $ \fQ_1 \subset \hfP $,
	we have
	\[ 
	\fQ_1 = \hfP.
	\] 
	This means that (\ref{decompP}) becomes			
	\begin{equation}
	\label{decompP2}
	\fP \hR = \hfP \cap \fQ_2 \cap \ldots \cap \fQ_s .
	\end{equation}
	
	\medskip

	Up to this point, we have not used the property that $ R $ is Henselian yet. 
	We claim that $ R $ Henselian implies $ s = 1 $: 
	First, $ A := R/P $ is a reduced Henselian excellent local ring and $ \Spec(A) $ is irreducible. 
	Thus, \cite{EGA_IV} Proposition~(18.6.12) 
	provides that $ A $ is unibranch 
	and
	the latter is equivalent to $ \widehat A $ being integral, by \cite{EGA_IV} Scholie~(7.8.3)(vii). 
	So, $ s = 1 $.
	
	In other words, \eqref{decompP2} is $ P \hR = \hP $. 
	With the notations of Proposition~\ref{Prop:EmptyCaseSpecial}, we get: $I_{\HS(\hX)} = \langle \, \hy \, \rangle \hR $. 

	Using the identical argument, 
	(starting from \eqref{eq:HSstratum}),
	the assumption that $ R $ is excellent finishes the proof of Theorem~\ref{Thm:coord_pol_hensel}.
\end{proof}

Again, let us emphasize that the previous proof provides the following result (without the assumption $ R $ Henselian)

\begin{Prop}
	\label{Prop:decomposition_P_hR}
	Let $ R $ be a regular local G-ring, $ J \subset R $ be a non-zero ideal and $ ( u , y ) = ( u_1, \ldots, u_e; y_1, \ldots, y_r ) $ be a regular system of parameters of $ R $  such that 
	$ (u) $ is a regular $ (R/J) $-sequence
	and 
	$ ( y ) $ determines the directrix of $ J' = J \cdot R' $, where $ R' = R / \langle u \rangle $.
	
	Suppose $ \cpoly Ju = \varnothing $.
	Let $ ( \hy ) = ( \hy_1 ,\ldots, \hy_r ) $ be the elements in $ \hR $ such that $ \poly Ju\hy = \varnothing $ 
	obtained by Hironaka's vertex preparation. 
	If we define 
	\[  
		\hfP := \langle \hy_1, \ldots, \hy_r \rangle \subset \hR 
		\ \ \
		\mbox{ and } 
		\ \ \
		\fP : = \hfP \cap R \subset R,
	\]
	then we have a unique decomposition
	\[
		\fP \hR = \hfP \cap \fQ_2 \cap \ldots \cap \fQ_s ,
	\]
	for certain primes ideal $ Q_2, \ldots, Q_s \subset \hR $.
\end{Prop}

\medskip 

\section{Empty Characteristic Polyhedron IV: Finding Generators}
\label{sec:empty_generators}

In this section, we explain how to obtain suitable generators $ (g_1, \ldots, g_m) $ for the ideal $ J \subset R $ such that 
$ \poly guz = \cpoly Ju = \varnothing $, 
where $ (z) = (z_1, \ldots, z_r) $ are the elements constructed in the previous subsections.
For this step, it is not necessary to assume that any of the hypotheses $ (*)$ or $ R $ Henselian hold.
(Recall that, if $(\Pol)$ holds, we obtained the appropriate generators in a different way).

\begin{Prop}
	\label{Prop:EmptyCaseSpecial_new}
	Let $ R $ be a regular local G-ring, 
	$ J \subset R $ be a non-zero ideal and 
	$ ( u , y ) = ( u_1, \ldots, u_e; y_1, \ldots, y_r ) $ be a regular system of parameters for $ R $ 
	such that 
	$ (u) $ is a regular $ (R/J) $-sequence and
	$ ( y ) $ determines the directrix of $ J' = J \cdot R' $, where $ R' = R / \langle u \rangle $. 
	Let $ (f) = (f_1, \ldots, f_m) $ be a $ 0 $-normalized $ ( u ) $-standard basis for $ J $.
	
	Suppose $ \cpoly Ju = \varnothing $.
	Let  $ ( \hy ) = ( \hy_1 ,\ldots, \hy_r ) $ be the elements in $ \hR $ obtained by Hironaka's vertex preparation such that $ \poly Ju\hy = \varnothing $.
	Assume there exist $(z) = (z_1, \ldots, z_r ) $ in $R$ such that $(u,z)$ is a regular system of parameters for $ R $ 
	with 
	$  
		\langle z \rangle \cdot \hR = \langle \hy \rangle .
	$ 
	
	Then we have $ z_j \in  I(\poly fuy ;1) $ (Definition \ref{Def:I(Delta,b)}), for $ 1 \leq j \leq r  $, and $ \poly fuz  \subset \poly fuy  $.
	Furthermore, there exists a $(u)$-standard basis $ (g) = (g_1, \ldots, g_m ) $ for $J$ in $ R $ 
	such that $ \poly guz  = \varnothing $ and 
	\[ 
		g_i = f_i + \sum_{a=1}^{i-1}  h_{i,a} f_a \in R ,
		\ \ \  
		h_{i,a} \in I( \poly fuy ; \nu_{i}-\nu_a), 
		\ \ 1\leq i\leq m.
	\] 
\end{Prop}

By Proposition \ref{Prop:EmptyCaseSpecial} (resp.~Theorem~\ref{Thm:coord_pol_hensel}), the assumptions of the above theorem hold true if hypothesis $ (*) $ is true (resp.~if $ R $ is Henselian).
Therefore, the mentioned proposition and the previous one imply Claim~\ref{claim:polyvide} if hypothesis $ (*) $ holds (resp.~if $ R $ is Henselian).
Together with Theorem~\ref{Thm:ReducEmpty}, we then obtain Theorems~\ref{MainThm:*} and~\ref{MainThm:Hensel}.

\begin{proof}
	Since the characteristic polyhedron is empty, we have 
	$ \nu_i = \ord_M(f_i) $, for all $ i \in \{ 1, \ldots, m \} $, 
	and the elements $ (y) = (y_1, \ldots, y_r) $ determine the directrix of $ J $.
	In particular, every standard basis for $ J $ is a $ ( u ) $-standard basis, which follows by \cite{CJS} Lemma~6.8.
	
	As $R \to \hR$ is faithfully flat, for any ideal $I\subset R$, we have $I\hR \cap R=I$.
	By Hironaka's construction (Theorem \ref{Thm:3.17}), we have $ \hy_j = y_j + \widehat{\varphi}_j $, for some $ \widehat{\varphi}_j\in \langle u \rangle \cdot \hR $, for every $ j \in \{ 1, \ldots, r \} $.
	In particular, 
	$ \langle \hy \rangle \subset I(\Delta(f;u;y);1) \cdot \hR $
	and therefore,
	\[
	 \langle z \rangle = \langle \hy \rangle \cap R  \subset  I(\Delta(f;u;y);1)\cdot \hR \cap R = I(\Delta(f;u;y);1).
	\] 
	This gives the first assertion.
	
	By Hironaka's construction, applied to $(f_1,\ldots,f_m)$ and $(y_1,\ldots,y_r)$, 
	we find $ (\hg) = (\hg_1, \ldots, \hg_m) \in \hR^m  $, with
	\begin{equation}
	\label{eq:hg}
	\hg_1=f_1\in \langle z \rangle^{\nu_1}, \ \ \
	{\hg_i}=f_i+\sum_{a=1}^{i-1} \widehat{h_{i,a}} \,f_a\in \langle z \rangle^{\nu_i} \hR = \langle \hy \rangle^{\nu_{i}},
	\end{equation}  
	for $ 2 \leq i \leq m $, 
	with 
	$ 
	\widehat{h_{i,a}}\in I(\Delta(f;u;y);\nu_{i}-\nu_a)\hR
	$ 
	and $\poly \hg u \hy=\varnothing$.
	
	As $ (f) $ is $ 0 $-normalized (which is stable under vertex preparation), we have 
	$ \ini_M(\hg_i) = \ini_M(f_i) $.
	Further,
	$
		\hg_i  
		\in  \langle  z \rangle ^{\nu_i}  \hR
	$, 
	so $\poly{\hg}uz=\poly{\hg}u{\hy}=\varnothing $.

	Since $ \hg_i \in \langle  z \rangle ^{\nu_i} \hR $, for all  $ i \in \{ 1, \ldots, m \} $,
		we obtain that $ \Spec(\hR/  \langle  z \rangle \hR) $ is permissible for $ \Spec(\hR/J \hR) $.
	Hence, since $ R \to \hR$ is faithfully flat, $ D := \Spec(R/\langle z \rangle ) $  is permissible for $ X := \Spec(R/J) $ at the origin $ x = \Spec(R/M) \in D $, by \cite{CJS} Theorem 2.2(3).
	In particular, $ X $ is normally flat along $ D $ at $ x $.  
	By \cite{CJS} Theorem~2.2(2)(iv), we can find $ (g') = (g'_1, \ldots, g'_m) \in R^m $ a standard basis of $J$,  such that,
	for $ 1 \leq i \leq m $, 
	\begin{equation}
	\label{eq:g'}
	g'_i=\sum_{a=1}^m  h'_{i,a}\, f_a\in \langle z \rangle^{\nu_i },
	\ \ \
	\mbox{for }
	h'_{i,a}\in R.
	\end{equation}
	As $(g')$ is a standard basis,  $\langle \ini_M (g') \rangle= \ini_M(J) = \langle \ini_M (f) \rangle$ 
	and we can adapt the basis $(g')$ to get
	\begin{equation}
	\label{eq:g'_ini}
		\ini_M(g'_i)= \ini_M(f_i ), 
		\ \ \ \mbox{ for } 1 \leq i \leq m. 
	\end{equation}
	In particular, we have $ \poly{g'}{u}{z} = \varnothing $ and $ g'_1 = f_1 = \hg_1 $, without loss of generality.
	By \eqref{eq:hg}, \eqref{eq:g'}, and \eqref{eq:g'_ini}, we have 
	$ \hg_i \equiv g'_i \mod M\langle z \rangle^{\nu_i} $, and we get, for all $i \in \{  2, \ldots, m \}  $,
	\[
	\sum_{a=1}^m h'_{i,a}\,f_a-f_{i}
	\ = \ 
	g'_i - f_i 
	\ \equiv \
	\hg_i - f_i 
	\ = \
	\sum_{a=1}^{i-1} \widehat{h_{i,a}}\, f_a
	\mod M\langle  z \rangle^{\nu_i} \hR .
	\]
	
	\noindent 
	As $\sum\limits_{a=1}^{i-1} \widehat{h_{i,a}} \, f_a\in 
	\sum\limits_{a=1}^{i-1}  f_a \cdot  I(\poly fuy;\nu_{i}-\nu_a)\hR$, by faithful flatness, 
	\[ 
	g'_i - f_i = \,
	\sum_{a=1}^m  h'_{i,a} f_a  - f_i  \in \sum_{a=1}^{i-1} f_a \, I(\Delta(f;u;y);\nu_i-\nu_a)+M\langle  z \rangle ^{\nu_i}.
	\] 
	Hence, for $2\leq i \leq m$, there is an expansion: 
	\[ 
	g'_{i} = 
	f_i + \sum_{a=1}^{i-1}  h_{i,a} f_a + k_i \in R, 
	\]
	\[
	\mbox{for some } \ k_i\in M\langle  z \rangle^{\nu_i}
	\  \mbox{ and } \  
	h_{i,a}\in I(\poly fuy ;\nu_i-\nu_a) \subset R.
	\] 
	
	\noindent 
	Since $ k_{i} \in M\langle  z \rangle ^{\nu_i} $, these $k_i$ cannot contribute to the polyhedron:
	$ \poly{g}{u}{z} =\poly{g'}{u}{z} = \varnothing $ with $g_1 := f_1$, $g_i:=f_{i}+\sum\limits_{a=1}^{i-1} h_{i,a} f_a$,  for $2\leq i \leq m$.
\end{proof}

%
%
%
%
%
%
%
%
%
%

\medskip

\section{On the General Case and More Examples}
\label{sec:general}

We end with some remarks on the general case. 
First, let us summarize which of the results are valid for any regular local $ G $-ring
(without assuming $ R $ to be Henselian or hypothesis $(*) $ or $ (\Pol) $ to hold):
\begin{itemize}
	\item 
	{\em Theorem~\ref{Thm:ReducEmpty}}:
	The reduction from a non-empty characteristic polyhedron to the case of an empty one.
	
	\smallskip 
	
	\item
	{\em Proposition~\ref{Prop:Pol_often_and_else_smaller_dim}}:
	Along a given face of $ \poly fuy \neq \varnothing $,
	the data either fulfills hypothesis $ (\Pol) $ (if $ \ell(f,u,y) > 0 $)
	or is contained in a smaller dimensional regular local $ G $-ring (if $ \ell(f,u,y) = 0 $).
	
	\smallskip 
	
	\item 
	{\em Proposition~\ref{Prop:decomposition_P_hR}}:
	Let $ (\hy) = (\hy_1, \ldots, \hy_r ) $ be the coordinates obtained by Hironaka's vertex preparation.
	Suppose that $ \cpoly Ju = \varnothing $.
	If we set $ \hP := \langle \hy \rangle \subset \hR $ and $ P := \hP \cap R \subset R $, 
	then we have a unique decomposition 
	$ 
		P\cdot \hR = \hP \cap Q_2 \cap \ldots \cap Q_s,
	$ 
	for $ Q_2, \ldots, Q_s \subset \hR $  prime ideals and $ s \geq 1 $.
	
	\smallskip 
	
	\item
	{\em Proposition~\ref{Prop:EmptyCaseSpecial_new}}:
	Once we have found appropriate coordinates $ (z) = (z_1, \ldots, z_r ) $ in $ R $ such that $ \poly Juz = \varnothing $, 
	we can find a $ (u) $-standard basis $ (g) = (g_1, \ldots, g_m ) $ for $ J $ such that $ \poly guz = \varnothing $.
\end{itemize}

\begin{Rk} 
An important step in the proof for the existence of $ (z) $ under hypothesis $ (*) $ or $ R $ Henselian, was to construct an ideal $ I \subset R $ with the property $ I \cdot \hR= \langle \hy \rangle $. 
The excellence of $ R $ then finished the proof. 
In case $ (*) $, this was obtain since 
$ I_{\HS(\hX)} = \langle \hy \rangle  $. 
In Example \ref{Ex:HS_morecomp}, we show that the this property does not hold, in general.

Nonetheless, if we can find a measure for the complexity of a singularity
(as an alternative to the Hilbert-Samuel function) defining a closed subscheme $ W \subset \Spec(R) $
such that the ideal $ I_W $ of $ W $ has the property $ I_W \cdot \hR = \langle \hy \rangle  $, the same arguments can be applied to finish the proof. 

One candidate to study is the extension of the Hilbert-Samuel function by the codimension of the ridge. 
In \cite{CPS}, a proof of the upper semi-continuity for this refinement is given, which is hidden in Giraud's work \cite{GiraudEtude}.
Hence, the locus $ W $, where it becomes maximal is closed.
Furthermore, using \cite{BernhardThesis} Main Theorem~C, one can deduce that the ideal 
of the maximal locus of this invariant in $\hR $ is $ \langle \hy \rangle $ if $ \dim ( X ) \leq 5 $.
Dietel introduces a refinement of the ridge and 
makes use Oda's classification of Hironaka schemes \cite{Oda} (see also \cite{BernhardThesis} section 10.9).
Due to increasing complexity, there is no classification of Hironaka schemes in higher dimension and more work is required.
\end{Rk}

We claim that the following result holds,
which characterizes $ \langle \hy \rangle $ if the characteristic polyhedron is empty. 
Unfortunately, we could not make significant use of it so far,
but we believe that it deserves to be mentioned.

\begin{Claim}
	Let $ R $ be a regular local G-ring, 
	$ J \subset R $ be a non-zero ideal and 
	$ ( u , y ) = ( u_1, \ldots, u_e; y_1, \ldots, y_r ) $ be a regular system of parameters for $ R $ 
	such that 
	$ (u) $ is a regular $ (R/J) $-sequence and
	$ ( y ) $ determines the directrix of $ J' = J \cdot R' $, where $ R' = R / \langle u \rangle $. 
	
	Suppose $ \cpoly Ju = \varnothing $.
	Let  $ ( \hy ) = ( \hy_1 ,\ldots, \hy_r ) $ be the elements in $ \hR $ obtained by Hironaka's vertex preparation such that $ \poly Ju\hy = \varnothing $.
	Then, $ \hD := \Spec ( \hR / \langle \hy \rangle  ) $ is the unique permissible center for $ \hX := \Spec ( \hR /  J\hR ) $ of maximal dimension,
	and every permissible center is contained in $ \hD $.
\end{Claim}

\begin{proof}[Proof (outline)]
	First of all, $ \hD $ is permissible.
	If there is a larger center containing $ \hD $, then we get a contradiction to the minimality of the number of generators of the directrix.
	If there is a permissible center that is transversal to $ \hD $, then we get again a contradiction to the property that the system $ ( \hy ) $ yields the directrix.
	Suppose there is a permissible center that is tangent to $ \hD $ and denote the corresponding ideal $ I' $.
	Then we get that the associated polyhedron cannot be empty and thus $ g \notin I'^n $, i.e., $ \Spec ( R / I')  $ is not a permissible center.
\end{proof}

\smallskip

	In general, $(*)(a)$ does only hold after a finite purely inseparable extension of the residue field $ k $.
	But by doing this the characteristic polyhedron may change drastically,
	as we illustrate in the following example.

	Furthermore, if differential operators map the ring $ R $ into itself, then one can deduce $(*)(a)$ by using them.
	But this is also not true in general 
	(see \cite{MatsumuraRing} Nomura's Theorem 30.6, p.~237).

\begin{Ex}
	\label{Ex:change_poly_closure}
Let $ R $ be a regular local ring containing a field $ k $ of characteristic $ p > 2 $ and set $ q = p^e $ for some $ e \in \IZ_{+} $.
Suppose $ (u_1, u_2, y_1, y_2) $ is a regular system of parameters for $ R $.
Consider the element
\[ 
f = y_1^q + \lambda y_2^q + \lambda u_1^{a q } + \lambda^2 u_2^{ b q },
\] 
where $ \lambda, \lambda^2 \in k \setminus k^q $ are $ q $-independent.
If we pass to the field extension $ k' := k[t] / \langle t^q - \lambda \rangle $ over $ k $, then
$ f  = y_1^q + t^q y_2^q + t^q u_1^{a q } + t^{2q} u_2^{ b q } $ in $ k' $.
Hence, hypothesis $(*)(a) $ holds and the directrix is given by $ z_0 := y_1 + t y_2 $.
If we set $ z := y_1 + t y_2 + t u_1^a + t^2 u_2^b $, we get $ f = z^q $ and the characteristic polyhedron over $ k' $ is empty.

Another way of deducing $(*)(a)$ is by applying the derivative $ \frac{\partial}{\partial \lambda} $. 
In this example minimizing the polyhedron for $ f $ is the same as minimizing the one of $ \langle f, \frac{\partial f}{\partial \lambda} \rangle $.
Further, we stay in the local ring $ R $ and get $ z_1 = y_1 $ and $ z_2 = y_2 + u_1^b $.
In $ R $ we cannot solve the vertex corresponding to the monomial $ \lambda^2 u_2^{ b q } $.
Hence the characteristic polyhedron is non-empty.

In conclusion, it is not clear how using purely inseparable extensions shall provide information on the original characteristic polyhedron. 
\end{Ex}

Here is another example, where the coordinates can be obtained using derivatives by constants, but this time with empty characteristic polyhedron. 

\begin{Ex}
	Let $ R $ be a regular local ring containing a field $ k $ of characteristic $ p > 0 $.
	Suppose $ (u_1, u_2, y_1, y_2) $ is a regular system of parameters for $ R $.
Consider the element
\[
f = y_1^p + \lambda y_2^p + u_1^{ 2p } + ( \lambda + 1 ) u_2^{ 2p },
\]	
where $ \lambda \in k \setminus k^p $.
By applying the derivative $ \frac{\partial}{\partial \lambda} $, we see that the desired elements are $ z_1 = y_1 + u_1^2 + u_2^2 $ and $ z_2 = y_2 + u_2^2 $.
We get $ f = z_1^p + \lambda z_2^p $. 
\end{Ex}

The following example (which is based on an example by Hironaka, see \cite{HiroAdd} Theorem 3, p.331) illustrates that $\Spec ( \hR / \langle \hy_1, \ldots, \hy_r \rangle ) $ is not necessarily equal to
the Hilbert-Samuel locus,
in general.
Moreover, it shows that the singular locus of the maximal Hilbert-Samuel locus does not characterize the ideal $ \langle \hy \rangle $.

\begin{Ex}
\label{Ex:HS_morecomp}
Consider the variety given by
\[
f = x^2 + \lambda y^2 + \mu z^2  + \lambda \mu w^2 + yz u^{11} 
\in R := k[x,y,z,w, u]_{\langle x,y,z,w, u \rangle },
\]
over a field $ k $, $ \car ( k ) = 2 $ and $ [k^2 (\lambda, \mu) : k^2 ] = 4 $.
The order at the origin is $ n = 2 $, the ideal of the directrix is given by $ \langle X, Y, Z, W \rangle $ and $ f \in \langle x, y, z, w \rangle^2 $.
The derivatives are
$ \frac{\partial f }{\partial y } = z u^{11},
\frac{\partial f }{\partial z } = y u^{11},
\frac{\partial f }{\partial \lambda } = y^2 + \mu w^2 $ 
and
$ \frac{\partial f }{\partial \mu } =  z^2  + \lambda w^2
$.
Therefore the locus of maximal order (which coincides with the maximal Hilbert-Samuel locus because we are considering a hypersurface) is
\[
V ( x, y, z,  w ) \cup V ( u, x^2 + \lambda \mu w^2, y^2 + \mu w^2 , z^2 + \lambda w^ 2 ).
\]	
Note that the singular locus of this is the origin $ V ( x,y,z,w,u) $.
\end{Ex}

\begin{Ex}
Let $ R $ be a regular local ring containing a field $ k $ of characteristic $ p > 0 $.
Suppose $ (u_1, u_2, y_1, y_2) $ is a regular system of parameters for $ R $.
Consider the element
\[
f = y_1^p + \lambda y_2^p + \lambda y_1^{p^2} + y_2^{p^2} + u_1^{p^3} + \lambda u_2^{p^3},
\]
where $ \lambda \in k \setminus k^p $.
A possible idea would be to introduce a weight on the coordinate $ y_2 $ such that we artificially create condition $(*)(a)$.
But if we do so, then we will never see that we have to solve $ y_2^{p^2 } $ because it will be in the interior of the corresponding polyhedron.

It is not hard to see that the characteristic polyhedron is empty and the desired parameters are $ z_1 := y_1 + y_2^p + u_1^{p^2} $ and $ z_2 = y_2 + y_1^p + u_2^{p^2} $.
\end{Ex}

%
%
%
%
%
%
%
%
%
%
%
%
%
%
%
%
%

\end{document}